\documentclass[11pt]{article}
\usepackage{amsmath,amsthm,amssymb}
\usepackage[hyperfootnotes=false]{hyperref}
\usepackage{color}
\usepackage{cancel}
\usepackage{titlesec}
\titleformat{\paragraph}
{\normalfont\normalsize\bfseries}{\theparagraph}{1em}{}
\titlespacing*{\paragraph}
{0pt}{3.25ex plus 1ex minus .2ex}{1.5ex plus .2ex}

\def\titlerunning#1{\gdef\titrun{#1}}
\makeatletter
\def\author#1{\gdef\autrun{\def\and{\unskip, }#1}\gdef\@author{#1}}
\def\address#1{{\def\and{\\\hspace*{18pt}}\renewcommand{\thefootnote}{}%
\footnote {#1}}%
\markboth{\autrun}{\titrun}}
\makeatother
\def\email#1{\hspace*{4pt}{\em e-mail}: #1}
\def\MSC#1{{\renewcommand{\thefootnote}{}%
\footnote{\emph{Mathematics Subject Classification (2020):} #1}}}
\def\keywords#1{\par\medskip
\noindent\textbf{Keywords:} #1}

\newtheorem{theorem}{Theorem}[section]

\newtheorem{prop}[theorem]{Proposition}
\newtheorem{cor}[theorem]{Corollary}
\newtheorem{lemma}[theorem]{Lemma}

\newtheorem{defin}[theorem]{Definition}

\newtheorem*{mtheorem}{Main Theorem}

\theoremstyle{definition}

\newtheorem{prob}[theorem]{Problem}

\newtheorem{remark}[theorem]{Remark}

\numberwithin{equation}{section}

\def\cB{\mathcal B}
\def\cC{\mathcal C}
\def\cD{\mathcal D}
\def\cE{\mathcal E}
\def\cF{\mathcal F}

\def\cH{\mathcal H}
\def\cI{\mathcal I}
\def\cL{\mathcal L}

\def\cO{\mathcal O}
\def\cP{\mathcal P}
\def\cQ{\mathcal Q}

\def\cR{\mathcal R}
\def\cS{\mathcal S}

\def\cX{\mathcal X}
\def\cY{\mathcal Y}
\def\cZ{\mathcal Z}
\def\PG{{\rm PG}}

\def\GF{{\rm GF}}

\def\PGL{{\rm PGL}}

\def\GL{{\rm GL}}

\def\s{\boldsymbol s}
\def\m{\boldsymbol m}

\begin{document}


\baselineskip=16pt

\titlerunning{}

\title{On cutting blocking sets and their codes}

\author{Daniele Bartoli
\and
Antonio Cossidente
\and
Giuseppe Marino 
\and 
Francesco Pavese}

\date{}

\maketitle

\address{
D. Bartoli: Dipartimento di Matematica e Informatica, Universit{\`a} di Perugia, Perugia, Italy; \email{daniele.bartoli@unipg.it}
\and 
A. Cossidente: Dipartimento di Matematica, Informatica ed Economia, Universit{\`a} degli Studi della Basilicata, Contrada Macchia Romana, 85100, Potenza, Italy; \email{antonio.cossidente@unibas.it}
\and
G. Marino: Dipartimento di Matematica e Applicazioni ``Renato Caccioppoli'', Universit{\`a} degli Studi di Napoli ``Federico II'', Complesso Universitario di Monte Sant'Angelo, Cupa Nuova Cintia 21, 80126, Napoli, Italy; \email{giuseppe.marino@unina.it}
\and
F. Pavese: Dipartimento di Meccanica, Matematica e Management, Politecnico di Bari, Via Orabona 4, 70125 Bari, Italy; \email{francesco.pavese@poliba.it}
}


\MSC{Primary 51E21; 94B05. Secondary 94B25; 51E20.}

\begin{abstract}
Let $\PG(r, q)$ be the $r$--dimensional projective space over the finite field $\GF(q)$. A set $\cX$ of points of $\PG(r, q)$ is a cutting blocking set if for each hyperplane $\Pi$ of $\PG(r, q)$ the set $\Pi \cap \cX$ spans $\Pi$. Cutting blocking sets give rise to saturating sets and minimal linear codes and those having size as small as possible are of particular interest. We observe that from a cutting blocking set obtained in \cite{FS}, by using a set of pairwise disjoint lines, there arises a minimal linear code whose length grows linearly with respect to its dimension. We also provide two distinct constructions: a cutting blocking set of $\PG(3, q^3)$ of size $3(q+1)(q^2+1)$ as a union of three pairwise disjoint $q$--order subgeometries and a cutting blocking set of $\PG(5, q)$ of size $7(q+1)$ from seven lines of a Desarguesian line spread of $\PG(5, q)$. In both cases the cutting blocking sets obtained are smaller than the known ones. As a byproduct we further improve on the upper bound of the smallest size of certain saturating sets and on the minimum length of a minimal $q$--ary linear code having dimension $4$ and $6$. 

\end{abstract}

\section{Introduction}
Let $q = p^h$, where $p$ is a prime and $h$ is a positive integer. Let $\PG(r, q)$ be the $r$--dimensional projective space over the finite field $\GF(q)$. We will denote by $(X_1,X_2,\ldots,X_{r+1})$ the homogeneous projective coordinates of a point of $\PG(r,q)$. A set $\cX$ of points of $\PG(r, q)$ is said to be a {\em $t$--fold blocking set} if every hyperplane $\Pi$ of $\PG(r, q)$ meets $\cX$ in at least $t$ points. 
\begin{defin}
A set of points $\cX \subset \PG(r, q)$ is a {\em cutting blocking set}, if for each hyperplane $\Pi$ of $\PG(r, q)$, the set $\Pi \cap \cX$ spans $\Pi$, i.e.,  $\Pi \cap \cX$ is not contained in any hyperplane of $\Pi$. 
\end{defin}
A cutting blocking set is said to be {\em minimal} if it is minimal with respect to set theoretical inclusion. If $\cX$ is a cutting blocking set, then any  hyperplane $\Pi$ of $\PG(r, q)$ is spanned by $\Pi \cap \cX$  and hence $|\Pi \cap \cX| \ge r$ and $\cX$ is an $r$--fold blocking set. The term {\em cutting blocking set} has been coined recently in \cite{BB}, see also \cite[Proposition 3.3]{ABN}, but such a substructure has been investigated earlier by several authors. In \cite{DGMP} a cutting blocking set is referred to as an {\em $r$--fold strong blocking set}, whereas in \cite{FS} a cutting blocking set is called a {\em generator set}. In \cite{DGMP}, some pointsets of $\PG(r, q)$, called $\rho$--saturating sets, were studied.
\begin{defin}
A set $\cY$ of points of $\PG(r, q)$ is said to be {\em $\rho$--saturating} if for any point $P \in \PG(r, q)$ there exist $\rho + 1$ points of $\cY$ spanning a subspace of $\PG(r, q)$ containing $P$, and $\rho$ is the smallest value with such property.
\end{defin}
In particular the authors pointed out that, by embedding $\PG(r, q)$ in $\PG(r, q^r)$, a cutting blocking set of $\PG(r, q)$ is an $(r-1)$--saturating set of $\PG(r, q^r)$ \cite[Theorem 3.2]{DGMP}. In this context, saturating sets of the smallest size are interesting as extremal objects and $\s_q(r, \rho)$ denotes the smallest size of a $\rho$--saturating set of $\PG(r, q)$. For more recent results on $\rho$--saturating sets of $\PG(r, q)$ the reader is referred to \cite{DMP, Denaux}. In \cite{FS} the authors, by investigating an idea introduced in \cite{HPT}, studied lines of $\PG(r, q)$ that are said to be in higgledy--piggledy arrangement. 
\begin{defin}
Let $\cL$ be a lineset of  $\PG(r, q)$. The lines of $\cL$ are said to be in {\em higgledy--piggledy arrangement} if the set of points covered by the lines of $\cL$ forms a cutting blocking set.
\end{defin}
Of particular interest is to look for the smallest size that a set consisting of lines in hyggledy--piggledy arrangement may have. It is not difficult to show that in $\PG(2, q)$ this number is three and it is known that in $\PG(3, q)$ the smallest set of lines in higgledy--piggledy arrangement has to contain four (pairwise disjoint) lines; see \cite{FS}. In $\PG(4, q)$, if $q$ is large enough, it is possible to find a set of six pairwise disjoint lines in higgledy--piggledy arrangement \cite[Proposition 12]{BKMP}. More generally, it is known that a set of lines of $\PG(r, q)$ in higgledy--piggledy arrangement has to contain at least $\lfloor\frac{r}{2}\rfloor + r$ lines, if $q \ge \lfloor\frac{r}{2}\rfloor + r$, and that there is a set of $2r-1$ pairwise disjoint lines with the required property whenever $q > 2r-1$; see \cite[Theorem 14, Theorem 20, Theorem 24]{FS}. It follows that if $q > 2r-1$, there exists a cutting blocking set in $\mathrm{PG}(r,q)$ of size at most $(2r-1)(q+1)$. Thus 
$$
\frac{r}{e}q + \frac{r-1}{2} \le \s_{q^r}(r, r-1) \le (2r-1)(q+1), 
$$
which improves on the known upper bounds for the size of the smallest saturating set. In the previous formula the lower bound arises from \cite[Lemma 4.0.1]{Denaux} and $e$ is the Euler's number. 

Here we deal with cutting blocking sets. 
The main achievements of this paper are summarized in the following theorem.
\begin{mtheorem}
\begin{enumerate}
    \item[(i)] In $\mathrm{PG}(3,q^3)$ there exists a minimal cutting blocking set of size $3(q+1)(q^2+1)$.
    \item [(ii)] In $\mathrm{PG}(5,q)$ there exists a minimal cutting blocking set of size $7(q+1)$.
\end{enumerate}
\end{mtheorem}
The paper is organized as follows. In Section \ref{sub} we construct a minimal cutting blocking set of $\PG(3, q^3)$ of size $3(q+1)(q^2+1)$ obtained by glueing together three suitable pairwise disjoint $q$--order subgeometries. In Section \ref{hp} we show that there is a set of seven lines of $\PG(5, q)$ in higgledy--piggledy arrangement. There arises a minimal cutting blocking set of $\PG(5, q)$ of size $7(q+1)$. As a byproduct we obtain the following improvements on the upper bound of the smallest size of a $2$--saturating set of $\PG(3, q^9)$ and of a $4$--saturating set of $\PG(5, q^5)$: 
\begin{align*}
& \frac{3}{e}q^3 + 1 \le \s_{q^9}(3, 2) \le 3(q+1)(q^2+1), \\
& \frac{5}{e}q + 2 \le \s_{q^5}(5, 4) \le 7(q+1). \\
\end{align*}

\subsection{Minimal and covering linear codes}

The concepts of cutting blocking set and saturating set are of interest not only from a geometrical point of view, but they also have applications in coding theory. For a vector $u = (u_1, \dots, u_n) \in \GF(q^n)$, its {\em support} is the set $supp(u) = \{i \;\; | \;\; u_i \ne 0\}$ and the {\em Hamming weight} $w(u)$ of $u$ is the cardinality of its support. The {\em Hamming distance} on $\GF(q)^n$ is defined as $d(u, v) = |supp(u - v)|$, for every pair of vectors $u, v \in \GF(q)$. A {\em $q$--ary linear code} of dimension $k$ and length $n$, or an $[n, k]_q$ code, $\cC$ is a $k$--dimensional vector subspace of $\GF(q)^n$. The elements of $\cC$ are called {\em codewords}. A {\em generator matrix} of $\cC$ is a matrix whose rows form a basis of $\cC$ as a vector space over $\GF(q)$. The {\em minimum distance of $\cC$} is $d = \min\{w(u) \;\; | \;\; u \in \cC, u \ne 0\}$. Let $u, v \in \GF(q)^n$. The vector $u$ is {\em $\rho$--covered by $v$} if $d(u, v) \le \rho$. The {\em covering radius} of a code $\cC$ is the smallest integer $\rho$ such that every vector of $\GF(q^n)$ if $\rho$--covered by at least one codeword of $\cC$. If $d$ or $\rho$ are needed, then $\cC$ is said to be an $[n, k, d]_q$ code or an $[n, k]_q \rho$ code. The weight distribution of $\cC$ is the sequence $A_0(\cC), \dots, A_n(\cC)$, where $A_i(\cC) = |\{u \in \cC \;\; | \;\; w(u) = i\}|$. For $u = (u_1, \dots, u_n), v = (v_1, \dots, v_n) \in \GF(q)^n$, let $u \cdot v = \sum_{i = 1}^n u_iv_i$ be the Euclidean inner product between $u$ and $v$. For a code $\cC$, its {\em dual code} is $\cC^\perp = \{v \in \GF(q)^n \;\; | \;\; v \cdot c = 0, \forall c \in \cC\}$. The dimension of the dual code $\cC^\perp$ or the codimension of $\cC$ is $n - k$. Any matrix which is a generator matrix of $\cC^\perp$ is called a {\em parity check matrix} of $\cC$. If $\cC$ is a linear $[n, k]_q \rho$ code, with parity check matrix $N$, its covering radius is the smallest $\rho$ such that every $w \in \GF(q)^{n - k}$ can be written as a linear combination of at most $\rho$ columns of $N$. For an introduction to coverings of vector spaces over finite fields and to the concept of code covering radius, see \cite{CHLL}.

The representatives of the points of a saturating set of $\PG(r, q)$ can be considered as columns of a parity check matrix of a $q$--ary linear code of codimension $r+1$. In particular, a $\rho$--saturating set of $\PG(r, q)$ of size $n$ corresponds to a parity check matrix of an $[n, n - (r+1)]_q (\rho+1)$ code; see \cite{DGMP, DMP} and references therein. 

Let $\cC$ be an $[n, k]_q$ linear code with generator matrix $M$. The code $\cC$ is called {\em non--degenerate} if there is no $i$, with $1 \le i \le n$, such that $u_i = 0$, for all $u \in \cC$. 
A non--zero codeword $u \in \cC$ is called {\em minimal} if every codeword $u' \in \cC$, with $supp(u') \subseteq supp(u)$ is a multiple of $u$ and $\cC$ is {\em minimal} if all its codewords are minimal. A minimal code $\cC$ is called {\em reduced} if for every $i$, with $1 \le i \le n$, the code obtained by deleting the same $i$--th coordinate in each codeword is not minimal. 
 
A {\em projective $[n, r+1, d]_q$ system} $\cZ$ is a set of $n$ points (counted with multiplicity) of $\PG(r, q)$ that do not all lie on a hyperplane and such that
$$
d = n - \max\{|H \cap \cZ|  \;\; : \;\; H \mbox{ hyperplane of } \PG(r, q)\}.
$$
There is a well--known correspondence between non--degenerate $[n, r+1, d]_q$ linear codes and projective $[n, r+1, d]_q$ systems. Indeed, the points of $\PG(r, q)$ represented by the columns of a generator matrix $M$ of an $[n, r+1]_q$ code form a set of $n$ points (counted with multiplicity) of $\PG(r, q)$. Viceversa, the code generated by the matrix having the representatives of the points of a projective $[n, r+1, d]_q$ system $\cZ$ as columns, gives an $[n, r+1]_q$ linear code. Moreover, for any non--zero vector $u = (u_1, u_2, \dots, u_{r+1}) \in \GF(q)^{r+1}$, the hyperplane of $\PG(r, q)$ with equation $u_1X_1 + u_2X_2 + \dots + u_{r+1} X_{r+1} = 0$ contains $|\cZ| - w$ points of $\cZ$ if and only if the codeword $u M$ has weight $w$. For more details the reader is referred to \cite{TV}.

In this setting it has been established a correspondence between $[n, r+1, d]_q$ minimal linear codes and projective $[n, r+1, d]_q$ systems that are cutting blocking sets. Furthermore reduced minimal $[n, r+1]_q$ linear codes are equivalent to minimal cutting blocking sets of $\PG(r, q)$ of size $n$; see \cite[Theorem 3.4]{ABN} and \cite[Theorem 14]{Tang}. 

In the context of minimal codes, one of the main issue is to provide explicit constructions of families of minimal codes of short length for a given dimension and in particular to construct minimal codes whose length grows linearly with respect to their dimension; see \cite[Problem 2]{ABN}, \cite{ABNR}, and \cite[Open Problem 24]{Tang}. In this regard, if $\m(r+1, q)$ denotes the minimum length of a minimal $q$--ary linear code having dimension $r+1$, from the construction of $2r-1$ pairwise disjoint lines of $\PG(r, q)$, $q > 2r-1$, in hyggledy--piggledy arrangement provided in \cite{FS}, the following upper bound can be derived:
$$
r(q + 1) \le \m(r+1, q) \le (2r - 1)(q + 1).
$$ 
The lower bound in the previous formula follows from \cite[Theorem 2.14]{ABNR}. Moreover, from the constructions of cutting blocking sets presented here, we obtain the following improvements:
\begin{center}
\begin{tabular}{rcll}
$3(q^3 + 1)\le$ &\hspace{-0.3 cm}$ \m(4, q^3) $& \hspace{-0.3 cm}$\le3(q + 1)(q^2+1)$,& see Theorem \ref{Th:mainPG3q^3}; \\
$5(q + 1)\le$ &\hspace{-0.3 cm}$\m(6, q) $& \hspace{-0.3 cm}$\le7(q + 1)$, & see Proposition \ref{Prop:mainPG5q}.
\end{tabular}
\end{center}

\section{Cutting blocking sets from subgeometries}\label{sub}

In this section we construct a cutting blocking set of $\PG(3, q^3)$ of size $3(q+1)(q^2+1)$ obtained by glueing together three suitable pairwise disjoint $q$--order subgeometries.

\subsection{Clubs and Splashes of $\PG(1, q^3)$}

Here we consider certain pointsets of $\PG(1, q^3)$, called clubs and splashes. These sets have been investigated by several authors \cite{BMZZ,BJ, BJ1, BJ2, DonatiDurante, FerretStorme, LV, LuMaPoTr2014, LuPo2004}. Before recalling their definitions and summarizing some of their properties, we mention the following well known results. 

\begin{lemma}[\cite{D}]\label{demp}
In $\PG(2, q^3)$, let $\pi_0$ be a $q$--order subplane and let $H$ be the stabilizer of $\pi_0$ in $\PGL(3, q^3)$. The group $H$ has three orbits on the points of $\PG(2, q^3)$: 
\begin{itemize}
\item $\pi_0$;
\item $\cO_2'$ consisting of the $q(q^2-1)(q^2+q+1)$ points of $\PG(2, q^3)$ lying on exactly one extended subline of $\pi_0$;
\item their complement $\cO_3'$ of size $q^3(q-1)(q^2-1)$. 
\end{itemize}
The group $H$ has three orbits on the lines of $\PG(2, q^3)$: 
\begin{itemize}
\item the $q^2+q+1$ extended sublines of $\pi_0$; 
\item $\cL_2'$ consisting of the $q(q^2-1)(q^2+q+1)$ lines of $\PG(2, q^3)$ intersecting $\pi_0$ in one point;
\item their complement $\cL_3'$ of size $q^3(q-1)(q^2-1)$ formed by lines disjoint from $\pi_0$.
\end{itemize}
\end{lemma}
Let $\ell_1 \in \cL_2'$, $\ell_2 \in \cL_3'$. Let $R_1$ be a point of $\cO_2'$, where $r$ is the unique extended subline of $\pi_0$ containing $R_1$, and let $R_2$ be a point of $\cO_3'$. In geometric terms, a club of $\PG(1, q^3)$ can be defined in two distinct ways.
\begin{itemize}
\item[{\em i)}] {\em By extension.} By extending the sublines of $\pi_0$ to lines of $\PG(2, q^3)$ and intersecting these with the line $\ell_1$, one obtains a {\em club of $\ell_1$ with head $\ell_1 \cap \pi_0$}.  
\item[{\em ii)}] {\em By projection.} By projecting the points of $\pi_0$ from $R_1$ onto a line $m$ not containing $R_1$, one obtains a {\em club of $m$ with head $m \cap r$}.
\end{itemize} 
A club has the following properties. 
\begin{itemize}
\item All clubs of $\PG(1, q^3)$ are projectively equivalent. 
\item A club of $\PG(1, q^3)$ contains $q(q+1)$ $q$--order sublines. 
\item Through two non--head points there passes exactly one $q$--order subline and it contains the head point. Two distinct $q$--order sublines of a club have at most a point in common distinct from the head point.  
\item Let $\cC$ be a club with head point $T$. The unique $q$--order subline determined by two non--head points and $T$ is contained in $\cC$.
\item A $q$--order subline of a club arises either from the sublines of $\pi_0$ through a point of $\pi_0$ or from a subline of $\pi_0$ according as the club is obtained by extension or by projection, respectively.
\end{itemize}
Similarly, it is possible to define geometrically a splash of $\PG(1, q^3)$ in two equivalent ways.
\begin{itemize}
\item[{\em i)}] {\em By extension.} By extending the sublines of $\pi_0$ to lines of $\PG(2, q^3)$ and intersecting these with the line $\ell_2$, one obtains a {\em splash of $\ell_2$}.  
\item[{\em ii)}] {\em By projection.} By projecting the points of $\pi_0$ from $R_2$ onto a line $m$ not containing $R_2$, one obtains a {\em splash of $m$}. 
\end{itemize} 
A splash has the following properties. 
\begin{itemize}
\item All splashes of $\PG(1, q^3)$ are projectively equivalent. 
\item A splash of $\PG(1, q^3)$ contains $2(q^2+q+1)$ $q$--order sublines divided into two equally sized families, say $\cF_1$ and $\cF_2$. 
\item A $q$--order subline of a family arises either from the extended sublines of $\pi_0$ through a point of $\pi_0$ or from a subline of $\pi_0$ according as the splash is obtained by extension or by projection, respectively. A $q$--order subline of the opposite family arises either from the extended sublines of $\pi_0$ of a dual subconic of $\pi_0$ or from a subconic of $\pi_0$ according as the splash is obtained by extension or by projection, respectively. 
\item Two distinct sublines of the same family meet in one point, whereas sublines of different families have in common $0, 1$ or $2$ points. Through two points there is exactly one $q$--order subline of each family.   
\item Let $s$ be a fixed $q$--order subline of $\cF_1$. Among the $q$--order sublines of $\cF_2$ there are $q+1$ meeting $s$ in one point, $q(q+1)/2$ meeting $s$ in two points, and $q(q-1)/2$ disjoint from $s$. 
\item The stabilizer of a splash in $\PGL(2, q^3)$ is a group of order $2(q^2+q+1)$ and it acts transitively on its $2(q^2+q+1)$ $q$--order sublines.
\end{itemize}

In the remaining part of this subsection we prove further properties regarding splashes of $\PG(1, q^3)$ that will be used in the paper.
 
Let $r_1$ be the canonical $q$--order subline of $\PG(1, q^3)$ and let $K = \PGL(2, q)$ be the stabilizer of $r_1$ in $\PGL(2, q^3)$. Let $x \in \GF(q^3) \setminus \GF(q)$ and let 
$$
\cS_x = \{(z - z^q, x z^q - x^q z) \;\; | \;\; z \in \GF(q^3) \setminus \{0\}\} = \{(t^q - t, x t - x^q t^q) \;\; | \;\; t \in \GF(q^3) \setminus \{0\}\}. 
$$

Since the projectivity of $\PGL(2, q^3)$ induced by 
$$
\begin{pmatrix}
x & 1 \\
x^q & 1 \\
\end{pmatrix}
$$
maps $\cS_x$ to the splash $\{(z, z^q) \;\; | \;\; z \in \GF(q^3) \setminus \{0\}\}$, we have that $\cS_x$ is a splash of $\PG(1, q^3)$. The stabilizer of $\cS_x$ in $\PGL(2, q^3)$ is generated by 
$$
\left\{ \begin{pmatrix}
\xi x - \xi^q x^q & \xi - \xi^q \\
-x^{q+1} (\xi - \xi^q) & \xi^q x - \xi x^q  \\
\end{pmatrix} \;\; \Big| \;\; \xi \in \GF(q^3) \setminus \{0\} \right\} \mbox{ and }
\begin{pmatrix}
-1 & 0 \\
x+x^q & 1 \\
\end{pmatrix}.
$$
If $z \in \GF(q)$ or $z = x + a$, where $a \in \GF(q)$, then $(z - z^q, x z^q - x^q z)$ is a point of $r_1$. Hence $r_1$ is a $q$--order subline of $\cS_x$. Let $\cF_1$ denote the family of $q$--order sublines of $\cS_x$ containing $r_1$ and let $\cF_2$ be the opposite family. Thus
$$
\cF_1 = \{r_{\xi} \;\; : \;\; \xi \in \GF(q^3) \setminus \{0\}\}, \qquad \cF_2 = \{r'_{\xi} \;\; : \;\; \xi \in \GF(q^3) \setminus \{0\}\}, 
$$ 
where
\begin{eqnarray*}
r_\xi &=& \left\{(\xi (x + a) - \xi^q (x^q + a), x \xi^q (x^q+a) - x^q \xi (x + a)) \;\; | \;\; a \in \GF(q)\right\} \cup \left\{(\xi - \xi^q, x \xi^q - x^q \xi)\right\},\\
r'_\xi &=& \left\{(\xi^q (x^q + a) - \xi (x + a), x \xi (x+a) - x^q \xi^q (x^q + a)) \;\; | \;\; a \in \GF(q)\right\} \cup \left\{(\xi^q - \xi, x \xi - x^q \xi^q)\right\}.
\end{eqnarray*}
\begin{lemma}\label{stab1}
Let $r'_\xi$ be a $q$--order subline of $\cF_2$ such that $|r'_\xi \cap r_1| = i$, $0 \le i \le 2$. Then there is a subgroup of $K$ if order $q+1-i$ fixing $r'_\xi$. 
\end{lemma}
\begin{proof}
Let $r'_\xi \in \cF_2$ such that $|r'_\xi \cap r_1| = 1$. 
\begin{itemize}
    \item If $r'_{\xi} \cap r_1 = \{(0, 1)\}$, then $\xi \in \GF(q)$ and $r'_\xi = r'_1$, where
$$
r'_1 = \{R_a = (1, -a-x-x^q) \;\; | \;\; a \in \GF(q)\} \cup \{R = (0, 1)\}. 
$$
Consider the subgroup $\{\gamma_c \;\; : \;\; c \in \GF(q)\}$ of $K$ of order $q$ fixing $R$, where $\gamma_c$ is induced by the matrix
$$
\begin{pmatrix}
1 & 0 \\
c & 1 \\
\end{pmatrix},
$$
for some $c \in \GF(q)$. Then $\gamma_c(R_a) = R_{a-c}$ and $\gamma_c$ fixes $r'_1$. 
\item If $r'_{\xi} \cap r_1 = \{(1, b)\}$, for a fixed $b \in \GF(q)$, then $\xi = \frac{1}{(x+b)^2}$ and 
$$r'_\xi = r'_{(x+b)^{-2}}=
 \{R_a  \;\; | \;\; a \in \GF(q)\}\ 
\cup \ \{R = (x + x^q + 2b, b^2 - x^{q+1})\},
$$
where 
$$R_a = (x^{q+1} + a(x + x^q) + 2ab-b^2, (2b-a) x^{q+1} + b^2 (x+x^q) + ab^2).$$
Consider the subgroup $\{\gamma_c \;\; : \;\; c \in \GF(q)\}$ of $K$ of order $q$ fixing $R_b$, where $\gamma_c$ is induced by the matrix
$$
\begin{pmatrix}
1 - bc & c \\
-b^2c & 1+bc \\
\end{pmatrix},
$$ 
for some $c \in \GF(q)$. Let $\gamma_c \ne id$, i.e., $c \ne 0$. Then $\gamma_c(R) = R_{\frac{bc-1}{c}}$ and, for $a \ne b$, we have that $\gamma_{c}(R_{a})$ equals $R$ or $R_{\frac{a - bc(a-b)}{1- c(a-b)}}$, according as $c = (a-b)^{-1}$ or $c \ne (a-b)^{-1}$. Hence $\gamma_c$ fixes $r'_{(x+b)^{-2}}$.
\end{itemize}

Suppose that $r'_\xi$ is such that $|r'_\xi \cap r_1| = 2$. 
\begin{itemize}
    \item If $r'_{\xi} \cap r_1 = \{(0, 1), (1, b)\}$, for a fixed $b \in \GF(q)$, then $\xi = (x + b)^{-1}$ and $r'_\xi = r'_{(x+b)^{-1}}$, where
$$
r'_{(x+b)^{-1}} = \{R_a = (a - b, x^{q+1} +b (x+x^q) + ab) \;\; | \;\; a \in \GF(q)\} \cup \{R = (1, b)\}. 
$$
Let $\Gamma$ be the subgroup of $K$ of order $q-1$ fixing both $R$ and $R_b$ induced by the subgroup of $\GL(2, q)$ given by
$$
\left\{ 
\begin{pmatrix}
c & 0 \\
(c - d) b & d \\
\end{pmatrix} \;\; : \;\; c, d \in \GF(q), cd \ne 0 \right\}.
$$
A member of $\Gamma$, say $\gamma_c$, is induced by the following matrix
$$
\begin{pmatrix}
1 & 0 \\
(1- c)b & c \\
\end{pmatrix},
$$
for some $c \in \GF(q) \setminus \{0\}$. Then $\gamma_c(R_a) = R_{\frac{a-b+bc}{c}}$ and $\gamma_c$ fixes $r'_{(x+b)^{-1}}$. 
\item If $r'_{\xi} \cap r_1 = \{(1, b_1), (1, b_2)\}$, for two fixed elements $b_1, b_2 \in \GF(q)$, then $\xi = \frac{1}{(x+b_1)(x+b_2)}$. In this case the subline $r'_\xi = r'_{(x+b_1)^{-1}(x+b_2)^{-1}}$ turns out to be
$$
\{R_a  \;\; | \;\; a \in \GF(q)\}\quad  \cup \quad \{R = (x + x^q + b_1 + b_2, b_1 b_2 - x^{q+1})\}, 
$$
where 
$$R_a=(x^{q+1} + a(x + x^q) + a(b_1+b_2) - b_1b_2, (b_1 + b_2 - a) x^{q+1} + b_1b_2 (x+x^q + a)).$$
Let $\Gamma$ be the subgroup of $K$ of order $q-1$ fixing both $R_{b_1}$ and $R_{b_2}$ induced by the subgroup of $\GL(2, q)$ given by
$$
\left\{ 
\begin{pmatrix}
c & d \\
- d b_1b_2 & c + d(b_1 + b_2) \\
\end{pmatrix} \;\; : \;\; c, d \in \GF(q), (c + b_1d)(c + b_2d) \ne 0 \right\}.
$$
A non--trivial projectivity of $\Gamma$, say $\gamma_c$, is induced by the following matrix
$$
\begin{pmatrix}
c & 1 \\
-b_1b_2 & c + b_1 + b_2 \\
\end{pmatrix},
$$
for some $c \in \GF(q) \setminus \{- b_1, - b_2\}$. Then $\gamma_c(R) = R_{-c}$ and $\gamma_{c}(R_{a})$ equals $R$ or $R_{\frac{ac + b_1b_2}{b_1+ b_2 + c - a}}$, according as $c = a - b_1 - b_2$ or $c \ne a - b_1 - b_2$. Hence $\gamma_c$ fixes $r'_{(x+b_1)^{-1}(x + b_2)^{-1}}$. 
\end{itemize}

Assume that $r'_\xi$ is such that $|r'_\xi \cap r_1| = 0$. Let $s_1, s_2$ be elements of $\GF(q)$ such that the polynomial $f(X) = X^2 + s_1 X + s_2$ is irreducible over $\GF(q)$. This means that there exists $u \in \GF(q^2) \setminus \GF(q)$ such that $f(u) = f(u^q) = 0$ and hence $u^{q+1} = s_2$, $u + u^q = - s_1$. In this case $\xi = \frac{1}{(x + u)(x + u^q)}$ and 
$$r'_\xi = r'_{(x+u)^{-1}(x+u^q)^{-1}}=
r'_{\frac{1}{(x+u)(x+u^q)}} = \{R_a  \;\; | \;\; a \in \GF(q)\} \
 \cup \ \{R = (x + x^q - s_1, s_2 - x^{q+1})\},$$
 where 
$$R_a = (x^{q+1} + a(x + x^q) - a s_1 - s_2, -(s_1 + a) x^{q+1} + s_2 (x+x^q + a)).$$
Let $\Gamma$ be the subgroup of $K$ of order $q+1$ induced by the subgroup of $\GL(2, q)$ given by
$$
\left\{ 
\begin{pmatrix}
c & d \\
- s_2 d & c - s_1 d \\
\end{pmatrix} \;\; : \;\; c, d \in \GF(q), (c, d) \ne (0, 0) \right\}.
$$
A non--trivial projectivity of $\Gamma$, say $\gamma_c$, is induced by the following matrix
$$
\begin{pmatrix}
c & 1 \\
-s_2 & c - s_1 \\
\end{pmatrix},
$$
for some $c \in \GF(q)$. Then $\gamma_c(R) = R_{-c}$ and $\gamma_{c}(R_{a})$ equals $R$ or $R_{\frac{ac + s_2}{c - a - s_1}}$, according as $c = a + s_1$ or $c \ne a + s_1$. Hence $\gamma_c$ fixes $r'_{(x+u)^{-1}(x + u^q)^{-1}}$. 
\end{proof}

\begin{lemma}\label{stab2}
No non--trivial projectivity of $K$ fixes $\cS_x$. 
\end{lemma}
\begin{proof}
Let $\gamma$ be the projectivity of $K$ associated with the matrix 
$$
\begin{pmatrix}
a & b \\
c & d \\
\end{pmatrix} \in \GL(2, q). 
$$
Then 
$$\gamma(t^q - t, x t - x^q t^q) = - \left((bx^q-a)t^q - (bx-a)t, \frac{c - dx}{bx-a}(bx-a)t - \frac{c-dx^q}{bx^q-a}(bx^q-a)t^q\right).$$ Hence $\gamma(\cS_x) = \cS_{\frac{c-dx}{bx-a}}$. Hence $\cS_{x} \ne \gamma(\cS_{x})$ unless $b = c = 0$ and $a = d$, i.e., $\gamma$ is the identity.
\end{proof}

\begin{prop}\label{characterization}
Let $\cS$ be a splash of $\PG(1, q^3)$ and let $s, s'$ be two $q$--order sublines of $\cS$ such that $|s \cap s'| = 1$. Then $s'$ belongs to the opposite family of $s$ if and only if $s'$ is stabilized by a subgroup of order $q$ of $Stab_{\PGL(2, q^3)}(s)$.
\end{prop}
\begin{proof}
Since $\PGL(2, q^3)$ is transitive on its splashes and the stabilizer of $\cS$ in $\PGL(2, q^3)$ is transitive on its $q$--order sublines, we may assume w.l.o.g. that $\cS = \cS_x$ and that $s = r_1$. If $s'$ belongs to the opposite family of $s$, then, from Lemma \ref{stab1}, we have that there exists a subgroup of order $q$ of $Stab_{\PGL(2, q^3)}(s)$ fixing $s'$. Viceversa, if $s'$ belongs to the same family as $s$, we claim that no non--trivial element of $Stab_{\PGL(2, q^3)}(s)$ fixes $s'$. Assume on the contrary that there exists such a projectivity $\gamma$. Then from Lemma \ref{stab2} we have that $\cS \ne \gamma(\cS)$, where $s, s'$ are two $q$--order sublines of both $\cS$ and $\gamma(\cS)$, contradicting \cite[Theorem 5.2]{BJ2}.  
\end{proof}

\begin{theorem}\label{splash}
There are $q^3-q$ splashes of $\PG(1, q^3)$ through a $q$--order subline.
\end{theorem}
\begin{proof}
Let $s$ be a $q$--order subline. In $\PG(1, q^3)$, there are $(q-1)(q+1)^2(q^2+1)$ $q$--order sublines intersecting $s$ in exactly one point. Let $\cB$ denote the set of such sublines. It can be easily checked that a subgroup of order $q$ of $Stab_{\PGL(2, q^3)}(s)$ acts on points of $\PG(1, q^3)$ by fixing a point $P$ of $s$ and by forming $q^2$ orbits of size $q$. Moreover, each of these orbits of size $q$ together with $P$ gives rise to a $q$--order subline of $\PG(1, q^3)$. Therefore there is a subset $\cB'$ of $\cB$ consisting of $(q^2-1)(q+1)$ $q$--order sublines that are stabilized by a subgroup of order $q$ of $Stab_{\PGL(2, q^3)}(s)$. 

Let us count in two ways the couples $(\cZ, s')$, where $\cZ$ is a splash of $\PG(1, q^3)$ containing $s$ and $s'$ is a further $q$--order subline of $\cZ$ belonging to the same family of $s$. From Proposition~\ref{characterization}, $s'$ belongs necessarily to $\cB \setminus \cB'$. Hence $s'$ can be chosen in $q^2(q+1)(q^2-1)$ ways, and from \cite[Theorem 5.2]{BJ2}, a splash containing $s$ and a further $q$--order subline $s'$ belonging to the same family of $s$ is uniquely determined. Hence on the one hand the number of these couples equals $q^2(q+1)(q^2-1)$. On the other hand, if $z$ denotes the number of splashes of $\PG(1, q^3)$ containing $s$, we have that the number of these couples equals $z (q^2+q)$. Therefore $z = q^3-q$, as required. 
\end{proof}

We are ready to prove the main result of this subsection.

\begin{theorem}\label{splash_main}
Let $\cS_1$, $\cS_2$ be two distinct splashes of $\PG(1, q^3)$ having a $q$--order subline $s$ in common. Then $\cS_1$, $\cS_2$ share a $q$--order subline belonging to the opposite family of $s$. 
\end{theorem}
\begin{proof}
We may assume w.l.o.g. that $\cS_1 = \cS_x$, where $x$ is a fixed element of $\GF(q^3) \setminus \GF(q)$ and that $s = r_1$. Let $s'$ be a $q$--order subline of $\cS_x$ belonging to the opposite family of $s$. From Lemma~\ref{stab1}, there is a subgroup $\Gamma_{s'}$ of $K$ of order $q + 1 - i$ fixing $s'$, where $|s \cap s'| = i$. Note that from Lemma~\ref{stab2}, a non--trivial element $\gamma$ of $\Gamma_{s'}$ maps $\cS_x$ to a splash of $\PG(1, q^3)$ distinct from $\cS_x$. Hence $\gamma(\cS_x)$ will share with $\cS_x$ both $s$ and $s'$. In particular $|\cS_x^{\Gamma_{s'}}| = q+1-i$. Moreover if $s', s''$ are two distinct $q$--order sublines of $\cS_x$ belonging to the opposite family of $s$, then $\cS_x^{\Gamma_{s'}} \cap \cS_x^{\Gamma_{s''}} = \{\cS_x\}$, otherwise if $\cS \in \cS_x^{\Gamma_{s'}} \cap \cS_x^{\Gamma_{s''}}$ and $\cS \ne \cS_x$, then $\cS$ would contain $s, s'$ and $s''$ and hence would share with $\cS_x$ at least $3q-2$ points, contradicting \cite[Theorem 23]{LV}. Since $s'$ can be chosen in $q+1$, $q(q+1)/2$ or $q(q-1)/2$ ways according as $|s \cap s'|$ equals $1$, $2$ or $0$, we have that there are $(q-1) \cdot (q+1) + (q-2) \cdot q(q+1)/2 + q \cdot q(q-1)/2 = q^3-q-1$ splashes of $\PG(1, q^3)$ distinct from $\cS_x$ such that each of them shares with $\cS_x$ the $q$--order subline $s$ and a further $q$--order subline belonging to the opposite family of $s$. The result now follows from Theorem~\ref{splash}.  
\end{proof}

\subsection{On the number of points covered by a special lineset of $\PG(2, q^3)$}

In $\PG(2, q^3)$, with the same notation used in Lemma \ref{demp}, let $\pi_0$ be a $q$--order subplane and let $\ell$ be a line of $\cL_3'$. The line $\ell$ contains $q^2+q+1$ points of $\cO_2'$ and $q^3-q^2-q$ points of $\cO_3'$. Moreover $\ell \cap \cO_2'$ is a splash $\cS$ of $\ell$ obtained by extending the $q^2+q+1$ sublines of $\pi_0$. Let $\cF_1$ be the family of $q$--order sublines of $\cS$ arising from the extended sublines of $\pi_0$ of certain dual subconics of $\pi_0$ and let $\cF_2$ be the family of $q$--order sublines of $\cS$ arising from the extended sublines of $\pi_0$ through a point of $\pi_0$. Let $s$ be a fixed element of $\cF_1$ and let $\cD$ be the set of extended sublines of the dual subconic of $\pi_0$ which gives rise to $s$. Then $\cD$ consists of $q+1$ extended sublines of $\pi_0$ such that through a point of $\PG(2, q^3)$ there pass at most $2$ lines of $\cD$ and through a point of $s$ there is exactly one line of $\cD$. Moreover two distinct lines of $\cD$ meet in a point of $\pi_0$. There are other $q^2(q+1)$ lines of $\PG(2, q^3)$ having at least a point in common with both $s$ and $\pi_0$. Let $\cE$ denote the set of such lines. In particular a line of $\cE$ has exactly one point in common with both $s$ and $\pi_0$, i.e., $\cE \subset \cL_2'$. Through a point of $s$ there are $q^2$ lines of $\cE$ and through a point of $\pi_0$ there are $q+1$,  $q$ or $q-1$ lines of $\cE$, according as there pass  $0$, $1$ or $2$ lines of $\cD$, respectively. A line of $\cD$ contains $q+1$ points of $\pi_0$ and $q^3-q$ points of $\cO_2'$, whereas a line of $\cE$ contains one point of $\pi_0$, $q^2$ points of $\cO_2'$ and $q^3-q^2$ points of $\cO_3'$. The main aim of this subsection is to bound the number of points of $\PG(2, q^3)$ lying on at least a line of $\cE$. The following result \cite[Lemma 5.6]{BJ1} will be useful.     

\begin{lemma}[\cite{BJ1}]\label{projection}
Let $r_1, r_2$ be lines of $\PG(2, q^3)$ and $b$ be a $q$--order subline of $r_2$ disjoint from $r_1$. Then each $q$--order subline of $r_1$, disjoint from $r_2$, is the projection of $b$ from exactly one point not on $r_1\cup r_2$.
\end{lemma}

\begin{lemma}\label{atmost2}
Through a point on a line of $\cD$ and not lying on $\pi_0 \cup s$, there passes at most one line of $\cE$.
\end{lemma}
\begin{proof}
Assume by contradiction that there is a point $P$ on a line $r \in \cD$, with $P \notin \pi_0 \cup s$, such that there are two lines of $\cE$, say $t_1$, $t_2$, passing through $P$. Let $T_i = t_i \cap \pi_0$, $i = 1,2$. The line $t$ obtained by joining $T_1$ and $T_2$ is an extended subline of $\pi_0$. Let $R = t \cap r$. Let $\cC_i$ be the set $t_i \cap (\cO_2' \cup \pi_0)$. Since $|t_i \cap \pi_0| = 1$, it follows that $\cC_i$ is a club of $t_i$ with head point $T_i$, $i = 1,2$. In particular $\cC_i$ is obtained by extending the sublines of $\pi_0$. Let $Q_i = t_i \cap s$ and let $m_i$ be the unique line of $\cD$ through $Q_i$, $i = 1,2$. Note that the club $\cC_i$ contains a unique $q$--order subline passing through $T_i, P, Q_i$, $i = 1,2$. Since every $q$--order subline of $\cC_i$ arises from the extended sublines through a point of $\pi_0$, we have that the lines $t, m_i, r$ form a pencil and $R = t \cap r \cap m_i$, $i = 1,2$. It follows that $m_1, m_2$ and $r$ are three lines of $\cD$ through the point $R \in \pi_0$, contradicting the fact that through a point of $\PG(2, q^3)$ there pass at most two lines of $\cD$. 
\end{proof}

\begin{lemma}\label{atmost3}
Through a point of $\cO_2'$ not lying on a line of $\cD$, there pass at most three lines of $\cE$.
\end{lemma}
\begin{proof}
Let $P$ be a point of $\cO_2'$, let $r$ be the extended subline of $\pi_0$ containing $P$ and let $R = r \cap \ell$. Since $P$ does not lie on a line of $\cD$, we have that $r \notin \cD$ and hence $R \notin s$. By projecting $\pi_0$ from $P$ onto $\ell$ we get a club $\cC$ of $\ell$ with head point $R$. Since $R \notin s$, the $q$--order subline $s$ is not contained in $\cC$. From \cite[Theorem 8]{LV}, we have that $|s \cap \cC| \le 3$. This means that there are at most three lines of $\cE$ passing through the point $P$. 
\end{proof}

\begin{prop}\label{rango2}
Let $A, B$ be two distinct points of $s$. There are at most $(q-1)(q^2-q+1)$ points of $\cO_2'$, not lying on a line of $\cD$, and contained in two lines of $\cE$ passing through $A$ and $B$.  
\end{prop}
\begin{proof}
Let $A$ and $B$ be two points of $s$, such that there are two lines of $\cE$, say $\ell_A$ and $\ell_B$, where $A \in \ell_A$, $B \in \ell_B$, $\ell_A \cap \ell_B = P \in \cO_2'$ and $P$ does not lie on a line of $\cD$. Note that there is a unique $q$--order subline $s'$ of $\cS$ distinct from $s$  and containing $A$ and $B$. In particular $s'$ belongs to the opposite family of $s$. Let $r_P$, $r_A$ and $r_B$ be the extended sublines of $\pi_0$ containing $P, A$ and $B$, respectively. Then $r_A, r_B \in \cD$ and $r_P \notin \cD$. Let $T = r_A \cap r_B \in \pi_0$, $T_A = \ell_A \cap \pi_0$, $T_B = \ell_B \cap \pi_0$ and let $m$ be the extended subline of $\pi_0$ passing through $T_A$ and $T_B$. Thus by projecting $m \cap \pi_0$ from $P$ onto $\ell$ the subline $s'$ is obtained. Furthermore an extended subline of $\pi_0$ through $T$ meets $\cS$ in a point of $s'$. 

We claim that $T \notin r_P$. Assume by contradiction that $T \in r_P$. If $T \notin m$, then by projecting the $q$--order subline $m \cap \pi_0$ from $T$ onto $\ell$ we get $s'$. On the other hand $s'$ is also obtained by projecting $m \cap \pi_0$ from $P$ onto $\ell$, where $P \ne T$, contradicting Lemma~\ref{projection}. Hence $T \in m$ and $m \cap \ell = R \in s'$. Since $P$ projects $m \cap \pi_0$ onto $s'$, we infer that the line joining $R$ with $P$ meets $\pi_0$ in a point of $m$. It follows that $P \in m = \langle T_A, T_B \rangle$ and $\ell_A = \langle T_A, P \rangle = \langle T_B, P \rangle = \ell_B$, a contradiction. 

We deduce that $T \notin r_P$ and $r_P \cap \ell \in \cS \setminus (s \cup s')$. By projecting $\pi_0$ from $A$ onto $r_P$ we get a club $\cC_A$ of $r_P$ with head point $T_1 = r_A \cap r_P$. Similarly, by projecting $\pi_0$ from $B$ onto $r_P$ we get a club $\cC_B$ of $r_P$ with head point $T_2 = r_B \cap r_P$. Since $T \notin r_P$, it follows that $T_1 \ne T_2$. Also $r_P \cap \pi_0$ is a $q$--order subline of both $\cC_A$, $\cC_B$ and $P \in \cC_A \cap \cC_B$. Hence $|\cC_A \cap \cC_B| \ge q+2$. We want to show that $r_P$ contains $q-1$ points of $\cO_2'$ contained in two lines of $\cE$ passing through $A$ and $B$. To this end it is enough to prove that $|\cC_A \cap \cC_B| = 2q$. Let $\bar{r}$ be the $q$--order subline of $r_P$ determined by $T_1, T_2, P$. Then $\bar{r}$ is a $q$--order subline of both $\cC_A$ and $\cC_B$. Since $\bar{r} \cap (r_P \cap \pi_0) = \{T_1, T_2\}$, we have that $|\cC_A \cap \cC_B| \ge 2q$. On the other hand if $Z$ were a point of $(\cC_A \cap \cC_B) \setminus (\bar{r} \cup (r_P \cap \pi_0))$, then the unique $q$--order subline of $r_P$ determined by $Z, T_1, T_2$ would lie in $\cC_A \cap \cC_B$ and $|\cC_A \cap \cC_B| \ge 3q - 2$, contradicting \cite[Theorem 23]{LV}. Therefore $|\cC_A \cap \cC_B| = 2q$.        

We have seen that if there exists a point $P \in \cO_2'$ not lying on a line of $\cD$ and contained in two lines of $\cE$ passing through $A$ and $B$, then $r_P \cap \ell \in \cS \setminus (s \cup s')$ and $r_P$ contains $q-1$ points with such a property. Since $|\cS \setminus (s \cup s')| = q^2-q+1$, the result follows.
\end{proof}

\begin{prop}\label{q+1}
Let $q >2$. There is a set $\cI$ consisting of $q^2$ points of $\cO_3'$ such that through a point of $\cI$ there are $q+1$ lines of $\cE$.
\end{prop}
\begin{proof}
Let $r$ be an extended subline of $\pi_0$ with $r \notin \cD$. Then $\bar{r} = r \cap \pi_0$ is a $q$--order subline. From Lemma \ref{projection}, there is a unique point $R$ of $\PG(2, q^3)$ with $R \notin r \cup \ell$ such that the $q$--order subline $s$ is obtained by projecting $\bar{r}$ from $R$ onto $\ell$. If $R$ were in $\pi_0$, then every extended subline of $\pi_0$ passing through $R$ would lie in $\cD$, contradicting the fact that through a point of $\pi_0$ there pass at most two lines of $\cD$. If $R$ were in $\cO_2'$, then either $R$ would lie on a line of $\cD$, contradicting Lemma \ref{atmost2} or $R$ would not lie on a line of $\cD$, contradicting Lemma \ref{atmost3}, whenever $q > 2$. 
\end{proof}

\begin{cor}\label{non-collinear}
If through a point of $\cO_3' \setminus \cI$ there pass three lines of $\cE$, then the points in common between these three lines and $\pi_0$ are not collinear.
\end{cor}
\begin{proof}
Assume that there is a point $P \in \cO_3'$ such that through $P$ there are three lines of $\cE$, say $t_1, t_2, t_3$, and that the three points $T_i = t_i \cap \pi_0$, $1 \le i \le 3$, are collinear. Let $Q_i = t_i \cap \ell$ and note that $Q_i \in s$, $1 \le i \le 3$. Then the $q$--order subline of $\pi_0$ containing $T_1, T_2, T_3$ is projected from $P$ onto the unique $q$--order subline of $\ell$ containing $Q_1, Q_2, Q_3$, namely $s$. The extended subline $r$ containing $T_1, T_2, T_3$ cannot belong to $\cal D$. Otherwise let $R = r \cap \ell$ and the line joining $P$ and $R$ meets $\pi_0$ in a point of $r \cap \pi_0$, i.e., $P \in r$, a contradiction. It follows that $P \in \cI$.
\end{proof}

\begin{prop}\label{atmost3_bis}
Through a point of $\cO_3'$ not lying in $\cI$, there pass at most three lines of $\cE$.
\end{prop}
\begin{proof}
Let $P$ be a point of $\cO_3' \setminus \cI$. By projecting $\pi_0$ from $P$ onto $\ell$, we get a splash $\cS'$ of $\ell$. If $\cS = \cS'$, then there are $q+1$ lines of $\PG(2, q^3)$ passing through $P$ and meeting both $s$ and $\pi_0$ in at least one point and hence $P \in \cI$, which is not the case. Hence $\cS \ne \cS'$. Assume by contradiction that there are at least four lines of $\cE$ through the point $P$. Then $|s \cap \cS'| \ge 4$.  By \cite[Theorem 8]{LV}, it follows that $s$ is a $q$--order subline of $\cS'$ as well. Since $\cS$ and $\cS'$ have in common $s$, from Theorem \ref{splash_main}, we have that $\cS$ and $\cS'$ have a further $q$--order subline $s'$ in common and $s'$ belongs to the opposite family of $s$. From Corollary \ref{non-collinear}, if we consider $s$ as a $q$--order subline of $\cS'$, then $s$ is obtained by projecting the points of a subconic of $\pi_0$ from $P$ onto $\ell$. Hence, when $s'$ is considered as a $q$--order subline of $\cS'$, it is obtained by projecting a subline $\bar{r}$ of $\pi_0$ from $P$ onto $\ell$. Similarly, since $s$, as a $q$--order subline of $\cS$, arises from the extended sublines of a dual subconic of $\pi_0$, we have that when $s'$ is considered as a  $q$--order subline of $\cS$, it is obtained by extending the sublines of $\pi_0$ through a point $T$ of $\pi_0$. Let $r$ be the line of $\PG(2, q^3)$ such that $r \cap \pi_0 = \bar{r}$ and let $R = r \cap \ell$. If $T$ were on $r$, then $R$ would belong to $s'$ and the line $m$ joining $R$ with $P$ would meet $\pi_0$ at a point of $\bar{r}$. Hence $m = r$ and $P \in \cO_2'$, a contradiction. Therefore $T \notin r$ and in particular $T \notin \bar{r}$. It follows that $s'$ is obtained by projecting $\bar{r}$ from $T$ onto $\ell$. On the other hand $s'$ is obtained by projecting $\bar{r}$ from $P$ onto $\ell$, with $P \ne T$. This contradicts Lemma \ref{projection}.        
\end{proof}

The results achieved in Lemma \ref{atmost2}, Lemma \ref{atmost3}, Proposition \ref{q+1} and Proposition \ref{atmost3_bis} can be summarized in the following theorem.

\begin{theorem}\label{plane1}
Through a point of $\PG(2, q^3)$, $q >2$, not lying in $\pi_0 \cup s$, there pass $0, 1, 2, 3$ or $q+1$ lines of $\cE$. In the last case, we get the $q^2$ points of $\cI \subset \cO_3'$. 
\end{theorem}

Let $z_i$ denote the number of points $P \in \cO_3'$ such that there are $i$ lines of $\cE$ through $P$, $i = 2, 3, q+1$ and let $z_j'$ be the number of points $P \in \cO_2' \setminus s$ such that there are $j$ lines of $\cE$ through $P$, $j = 2, 3$. We have that $z_{q+1} = q^2$. Let us count in two ways the pairs $(r, r')$, where $r, r' \in \cE$ and $r \cap r' \in (\cO_2' \cup \cO_3') \setminus s$. For a fixed $r \in \cE$, let $R = r \cap \pi_0$. There are $q^3-q+i$ lines of $\cE$ intersecting $r$ in a point of $(\cO_2' \cup \cO_3') \setminus s$ according as through the point $R$ there pass $i$ lines of $\cD$, $i = 0, 1, 2$. There are $q+1$ points of $\pi_0$ incident with one line of $\cD$, $q(q+1)/2$ points of $\pi_0$ incident with $2$ lines of $\cD$ and $q(q-1)/2$ points of $\pi_0$ on no line of $\cD$. Hence on the one hand the number of these couples equals
$$
\frac{q(q-1)}{2} \cdot (q+1) \cdot (q^3-q) + \frac{q(q+1)}{2} \cdot (q-1) \cdot (q^3-q+2) + q \cdot (q+1) \cdot (q^3-q+1). 
$$
On the other hand, the number of these couples turns out to be $2(z_2 + z_2') + 6(z_3 + z_3') + q^2 \cdot q(q+1)$. Comparing these two quantities, we have that
\begin{equation}\label{double_count1}
z_2 + z_2' + 3(z_3 + z_3') = \frac{q^2(q+1)(q^3-2q+1)}{2}. 
\end{equation}
Analogously, let us count in two ways the pairs $(r, r')$, where $r, r' \in \cE$ and $r \cap r' \in \cO_2' \setminus s$. Let $r, r' \in \cE$, with $r \ne r'$ and let $P = r \cap r'$. First note that from Lemma \ref{atmost2}, if $P \in \cO_2' \setminus s$, then $P$ does not lie on a line of $\cD$. Moreover from Proposition \ref{rango2}, for two fixed points $A, B \in s$, there are at most $(q-1)(q^2-q+1)$ points $P \in \cO_2' \setminus s$ such that $r \cap s = A$, $r' \cap s = B$ and $r \cap r' = P$. Therefore we have that $2z_2' + 6z_3' \le (q-1)(q^2-q+1) \cdot q(q+1)$, that is
\begin{equation}\label{double_count2}
z_2' + 3z_3' \le \frac{(q^3-q)(q^2-q+1)}{2}.
\end{equation}

\begin{prop}\label{bound}
If $q \ge 5$, then $z_2 + 2z_3 > q^5 - 2q^3$.
\end{prop}
\begin{proof}
Taking into account \eqref{double_count1} and \eqref{double_count2}, we have that $z_2 + 3z_3 \ge (q^3-q)(q^3-1)/2$. On the other hand, from Corollary \ref{non-collinear}, we have that the value $z_3$ cannot exceed the number of triangles of $\pi_0$, that is $q^3(q+1)(q^2+q+1)/6$. Hence 
\begin{equation}
\begin{split}
z_2 + 2 z_3 & \ge  \frac{(q^3-q)(q^3-1)}{2} - z_3 \ge \frac{(q^3-q)(q^3-1)}{2} - \frac{q^3(q+1)(q^2+q+1)}{6} \\ \nonumber
&  =  \frac{2q^6-2q^5-5q^4-4q^3+3q}{6}  > q^5 - 2q^3,
\end{split}
\end{equation}
whenever $q \ge 5$.
\end{proof}

\subsection{Cutting blocking sets of $\PG(3, q^3)$ as union of three $q$--order subgeometries}

Let $\PG(3, q^3)$ be the three--dimensional projective space over $\GF(q^3)$.

\begin{lemma}\label{intersection}
A plane of $\PG(3, q^3)$ shares with a $q$--order subgeometry of $\PG(3, q^3)$ either one point or $q+1$ points of a $q$--order subline or $q^2+q+1$ points of a $q$--order subplane. 
\end{lemma}
\begin{proof}
Let $\Sigma$ be a $q$--order subgeometry of $\PG(3, q^3)$, let $P$ be a point of $\Sigma$ and let $\pi$ be a plane of $\PG(3, q^3)$ such that $\pi_0 = \pi \cap \Sigma$ is a $q$--order subplane of $\pi$, with $P \notin \pi$. A plane $\sigma$ of $\PG(3, q^3)$ containing $P$ intersects $\pi$ in a line $r$ and $|\sigma \cap \Sigma|$ equals $1$, $q+1$ or $q^2+q+1$, according as $r \cap \pi_0$ equals $0$, $1$ or $q+1$, respectively. From Lemma \ref{demp}, in $\pi$ there are $q(q^2-1)(q^2+q+1)$ lines intersecting $\pi_0$ in one point and $q^3(q-1)(q^2-1)$ lines disjoint from $\pi_0$. Hence apart from the $(q+1)(q^2+1)$ planes of $\PG(3, q^3)$ intersecting $\Sigma$ in a $q$--order subplane, there are $q(q^2+q+1)(q^4-1)$ planes of $\PG(3, q^3)$ meeting $\Sigma$ in a $q$--order subline and $q^3(q^2-1)(q^4-1)$ planes of $\PG(3, q^3)$ having in common with $\Sigma$ exactly one point.
\end{proof}

We will refer to a line or a plane of $\PG(3, q^3)$ intersecting a $q$--order subgeometry $\Sigma$ in $q+1$ or $q^2+q+1$ points as a line or a plane of $\Sigma$, respectively. Let $\Sigma_1 = \PG(3, q)$ be the canonical $q$--order subgeometry embedded in $\PG(3, q^3)$.  Let $G = \PGL(4, q)$ be the stabilizer of $\Sigma_1$ in $\PGL(4, q^3)$ and let $\iota$ be the collineation of order three of $\PG(3, q^3)$ fixing pointwise $\Sigma_1$. 

\begin{lemma}
The group $G$ has three orbits on points of $\PG(3, q^3)$: 
\begin{itemize}
\item $\Sigma_1$;
\item $\cO_2$ of size $q(q^2+q+1)(q^4-1)$ consisting of points lying on exactly one line of $\Sigma_1$; 
\item $\cO_3$ of size $q^3(q^2-1)(q^4-1)$.
\end{itemize}
\end{lemma}
\begin{proof}
Let $\pi$ be a plane of $\Sigma_1$ and let $G_{\pi}$ be the stabilizer of $\pi$ in $G$. The group $G$ is transitive on the $(q+1)(q^2+1)$ planes of $\Sigma_1$, hence the three $G_{\pi}$--orbits on points of $\pi$ give rise to $\Sigma_1$, $\cO_2$, $\cO_3$. To compute their size, note that two planes of $\Sigma_1$ have in common $q+1$ points of $\Sigma_1$ and $q^3-q$ points of $\cO_2$ and that a point of $\cO_2$ lies on exactly $q+1$ planes of $\Sigma_1$. 
\end{proof}

Note that a point $P \in \PG(3, q^3) \setminus \Sigma_1$ belongs to $\cO_2$ or $\cO_3$ according as the points $P, P^{\iota}, P^{\iota^2}$ span a line or a plane of $\Sigma_1$, respectively. 

\begin{lemma}\label{lines}
The group $G$ has five orbits on lines of $\PG(3, q^3)$: 
\begin{itemize}
\item $\cL_1$ of size $(q^2+1)(q^2+q+1)$ consisting of lines of $\Sigma_1$. A line of $\cL_1$ has $q+1$ points in common with $\Sigma_1$ and $q^3-q$ points in common with $\cO_2$;
\item $\cL_2$ of size $q(q+1)(q^2+q+1)(q^4-1)$ consisting of lines meeting $\Sigma_1$ in one point and contained in a plane of $\Sigma_1$. A line of $\cL_2$ consists of one point of $\Sigma_1$, $q^2$ points of $\cO_2$ and $q^3-q^2$ points of $\cO_3$;
\item $\cL_3$ of size $q^3(q^2-1)(q^4-1)$ consisting of lines meeting $\Sigma_1$ in one point and not contained in a plane of $\Sigma_1$. A line of $\cL_3$ has one point in common with $\Sigma_1$ and $q^3$ points in common with $\cO_3$;
\item $\cL_4$ of size $q^3(q^2-1)(q^4-1)$ consisting of lines disjoint from $\Sigma_1$ and contained in a plane of $\Sigma_1$. A line of $\cL_4$ consists of $q^2+q+1$ points of $\cO_2$ and $q^3-q^2-q$ points of $\cO_3$;
\item $\cL_5$ of size $q^5(q^3-1)(q^4-1)$ consisting of lines disjoint from $\Sigma_1$ and not contained in a plane of $\Sigma_1$. A line of $\cL_5$ has $q+1$ points in common with $\cO_2$ and $q^3-q$ points in common with $\cO_3$.
\end{itemize}
\end{lemma}
\begin{proof}
Let $\pi$ be a plane of $\Sigma_1$ and let $G_{\pi}$ be the stabilizer of $\pi$ in $G$. Under the action of $G$ the three $G_{\pi}$--line orbits give rise to $\cL_1$, $\cL_{2}$ and $\cL_{4}$, respectively. Note that a line of $\cL_2$ has $q^2$ points of $\cO_2$ and $q^3-q^2$ points of $\cO_3$, whereas a line belonging to $\cL_4$ has $q^2+q+1$ points of $\cO_2$ and $q^3-q^2-q$ points of $\cO_3$. 

Let $\cQ$ be a quadratic cone of $\PG(3, q^3)$ such that $\cQ \cap \Sigma_1$ is a quadratic cone of $\Sigma_1$. Thus the vertex of $\cQ$, say $V$, belongs to $\Sigma_1$ and $\iota$ stabilizes $\cQ$. If $r$ is a line of $\cQ$ such that $r \cap \Sigma_1 = \{V\}$ and $P$ is a point of $r$, with $P \ne V$, then $P \in \cO_3$, otherwise the three lines $r$, $r^{\iota}$, $r^{\iota^2}$ of $\cQ$ would lie on a plane, a contradiction. This means that $|r \cap \cO_3| = q^3$ and that $r$ cannot be contained in a plane of $\Sigma_1$. On the other hand, if $\pi$ is a plane of $\Sigma_1$ and $V \notin \pi$, then $\pi \cap \cQ$ is a non--degenerate conic $\cQ(2, q^3)$ and $\cQ(2, q^3) \cap \Sigma_1$ is a non--degenerate subconic $\cQ(2, q)$ of $\cQ(2, q^3)$ contained in $\pi \cap \Sigma_1$. Let $R = \pi \cap r$. Then $\pi = \langle R, R^{\iota}, R^{\iota^2} \rangle$ and $R, R^\iota, R^{\iota^2} \in \cQ(2, q^3)$. Note that there are exactly $q^2+q+1$ non--degenerate conics of $\pi$ passing through $R, R^\iota, R^{\iota^2}$ and intersecting $\pi \cap \Sigma_1$ in a non--degenerate subconic. This set of $q^2+q+1$ conics gives rise to the so called circumscribed bundle of $\pi \cap \Sigma_1$; see \cite{BBEF}. It follows that the line $r$ is contained in exactly $q^2+q+1$ quadratic cones of $\PG(3, q^3)$ such that their intersection with $\Sigma_1$ is a quadratic cone of $\Sigma_1$. Since the stabilizer of $\cQ$ in $G$ is transitive on the $q^3-q$ lines of $\cQ$ meeting $\Sigma_1$ exactly in $V$ and $G$ is transitive on the $q^2(q^3-1)(q^2+1)(q+1)$ quadratic cones of $\Sigma_1$ \cite[Section 15.3]{H2}, we have that $r^G = \cL_3$.

Let $\cH$ be a hyperbolic quadric of $\PG(3, q^3)$ such that $\cH \cap \Sigma_1$ is a hyperbolic quadric of $\Sigma_1$. Thus $\cH$ contains $2(q^3+1)$ lines of $\PG(3, q^3)$ on two reguli, say $R_1$ and $R_2$. Among these $2(q^3+1)$ lines, there are $2(q+1)$ that belong to $\cL_1$ and that are on two reguli of $\Sigma_1$, say $\bar{R}_1 \subset R_1$ and $\bar{R}_2 \subset R_2$. If $t$ is a line of $R_1 \setminus \bar{R}_1$ and $P$ is a point of $t$, then $P \in \cO_2$ if and only if $P$ is on a line of $\bar{R}_2$. Similarly if $t \in R_2 \setminus \bar{R}_2$. This means that $|t \cap \cO_2| = q+1$, $|t \cap \cO_3| = q^3-q$ and that $t$ cannot be contained in a plane of $\Sigma_1$. Moreover the line $t$ is contained in exactly one hyperbolic quadric of $\PG(3, q^3)$ such that $\cH \cap \Sigma_1$ is a hyperbolic quadric of $\Sigma_1$, otherwise the three lines $t$, $t^{\iota}$, $t^{\iota^2}$ would lie on two distinct reguli of $\PG(3, q^3)$, a contradiction. Since the stabilizer of $\cH$ in $G$ is transitive on the $2(q^3-q)$ lines of $\cH$ disjoint from $\Sigma_1$ and $G$ is transitive on the $q^4(q^3-1)(q^2+1)$ reguli of $\Sigma_1$ \cite[Section 15.3]{H2}, we have that $t^G = \cL_5$.
\end{proof}

\begin{lemma}\label{planes}
The number of planes of $\PG(3, q^3)$ intersecting $\Sigma_1$ in at least $q+1$ points and passing through a line $\ell$ of $\PG(3, q^3)$, with $|\ell \cap \Sigma_1| \le 1$, equals either $q^2+1$, or $q^2+q+1$, or $1$, or $q+1$ according as $\ell$ belongs either to $\cL_2$, or $\cL_3$, or $\cL_4$, or $\cL_5$, respectively.
\end{lemma}
\begin{proof}
Let $\ell$ be a line of $\PG(3, q^3)$ such that $\ell$ is not a line of $\Sigma_1$. Then there is one or no plane of $\Sigma_1$ containing $\ell$ just as $\ell \in \cL_2 \cup \cL_4$ or $\ell \in \cL_3 \cup \cL_5$. Moreover, if $\ell \in \cL_2 \cup \cL_4$, then $|\ell \cap \Sigma_1| = 1$, whereas if $\ell \in \cL_3 \cup \cL_5$, then $|\ell \cap \Sigma_1| = 0$. It follows that if $\ell \in \cL_2$ there are $q^2$ planes of $\PG(3,q^3)$ containing $\ell$ and meeting $\Sigma_1$ in $q+1$ points, whereas if $\ell \in \cL_3$ there are $q^2+q+1$ planes of $\Sigma_1$ containing $\ell$ and sharing $q+1$ points with $\Sigma_1$. If $\ell \in \cL_4$, let $\pi$ be the unique plane of $\Sigma_1$ containing $\ell$. Through $\ell$ there pass $q^3$ planes distinct from $\pi$ and $|\Sigma_1 \setminus \pi| = q^3$. Since every plane has at least a point in common with $\Sigma_1$, we have that necessarily every plane through $\ell$ distinct from $\pi$ has exactly one point in common with $\Sigma_1$. If $\ell \in \cL_5$, from the proof of Lemma \ref{lines}, $\ell$ is contained in a unique hyperbolic quadric $\cH$ of $\PG(3, q^3)$ such that $\cH \cap \Sigma_1$ is a hyperbolic quadric of $\Sigma_1$. Hence there are at least $q+1$ planes containing $\ell$ and meeting $\Sigma_1$ in $q+1$ points. On the other hand, since there are other $q^3-q$ planes through $\ell$ and $|\Sigma_1 \setminus (\Sigma_1 \cap \cH)| = q^3 - q$, every other plane through $\ell$ has to share with $\Sigma_1$ exactly one point. 
\end{proof}

Let $S$ be a Singer group of $\Sigma_1$. Then $S$ is a subgroup of order $(q+1)(q^2+1)$ of a Singer group $\bar S$ of $\PG(3, q^3)$. Let $\bar S'$ be the unique subgroup of $\bar S$ of order $(q^2-q+1)(q^4-q^2+1)$. Thus a non--trivial element of $\bar S'$ maps $\Sigma_1$ to a $q$--order subgeometry of $\PG(3, q^3)$ distinct from $\Sigma_1$. Since $\bar S$ acts regularly on the points of $\PG(3, q^3)$ and $\bar S = \langle S, \bar S' \rangle$, we have that necessarily these $(q^2-q+1)(q^4-q^2+1)$ $q$--order subgeometries so obtained are pairwise disjoint. Hence they form a partition $\cP$ of the points of $\PG(3, q^3)$ into $q$--order subgeometries. Since $S$ is a subgroup of $G$, it follows that a $q$--order subgeometry of $\cP$ and distinct from $\Sigma_1$ consists either of points of $\cO_2$ or of points of $\cO_3$. In particular there are $q^3(q-1)(q^2-1)$ members of $\cP$ consisting of points of $\cO_3$ and hence $q^4-q$ members of $\cP$ formed by points of $\cO_2$. Recall the following results.

\begin{lemma}[\cite{Bruen}]\label{bruen}
No plane of $\PG(3, q^3)$ meets two distinct members of $\cP$ in $q^2+q+1$ points. 
\end{lemma}

\begin{lemma}[\cite{Drudge}, \cite{Glynn}]
Under the action of $S$, the lines of $\cL_1$ are partitioned into $q+1$ orbits:
\begin{itemize}
\item one orbit consisting of $q^2+1$ pairwise disjoint lines;
\item $q$ orbits $\cR_1, \dots, \cR_q$ each of size $(q+1)(q^2+1)$. Through a point of $\Sigma_1$ there pass $q+1$ lines of $\cR_i$ no three in a plane of $\Sigma_1$ and a plane of $\Sigma_1$ contains $q+1$ lines of $\cR_i$ no three through a point. 
\end{itemize} 
\end{lemma}

Note that if two lines of $\Sigma_1$ meet, then necessarily their intersection point belongs to $\Sigma_1$. Let $\ell$ be a line of $\cR_1$ and let $L$ be a point of $\ell \cap \cO_2$. Set $\Sigma_2 = L^S$. We will show that there exists a $q$--order subgeometry of $\cP$ consisting of points of $\cO_3$ that together with $\Sigma_1 \cup \Sigma_2$ forms a cutting blocking set.  

\begin{prop}\label{sec}
There are exactly $(q^3+q^2+1)(q+1)(q^2+1)$ lines of $\PG(3, q^3)$ having at least a point in common with $\Sigma_1$ and $\Sigma_2$. In particular $(q+1)(q^2+1)$ of these are lines of $\Sigma_1$ and $q^2(q+1)^2(q^2+1)$ are lines of $\cL_2$ meeting both $\Sigma_1$ and $\Sigma_2$ in one point. 
\end{prop}
\begin{proof}
Since $|\Sigma_2| = |L^S| = |\ell^S| = |\cR_1|$ and the unique line of $\cR_1$ through $L$ is $\ell$, necessarily $\ell$ meets $\Sigma_2$ in $L$. Hence every line of $\cR_1$ has exactly one point in common with $\Sigma_2$. Moreover no other line of $\Sigma_1$ may have a point in common with $\Sigma_2$. Therefore a line of $\PG(3, q^3)$ not belonging to $\cR_1$ and having at least a point in common with $\Sigma_1$ and $\Sigma_2$ belongs to either $\cL_2$ or $\cL_4$. Let $\pi$ be a plane of $\Sigma_1$ with $\ell \subset \pi$. Since there are exactly $q+1$ lines of $\cR_1$ contained in $\pi$ and every line of $\cR_1$ has exactly one point in common with $\Sigma_2$, we have that $|\pi \cap \Sigma_2| \ge q+1$. From Lemma~\ref{bruen}, we have that $\pi \cap \Sigma_2$ consists of $q+1$ points of a $q$--order subline, say $s$. Through the point $L$ of $s$ there pass $q^2$ lines of $\cL_2$ that are contained in $\pi$. Varying the plane $\pi$ among the $q+1$ planes of $\Sigma_1$ containing $\ell$, we have that there are $q^2(q+1)$ lines of $\cL_2$ through $L$ meeting both $\Sigma_1$ and $\Sigma_2$ in one point. Since $q^3+q^2 = |\Sigma_1 \setminus \ell|$, we have that every line through $L$ distinct from $\ell$ and intersecting $\Sigma_1$ in at least one point has exactly one point in common with $\Sigma_1$. Since $\Sigma_2 = L^S$, it follows that there are exactly $q^2(q+1)^2(q^2+1)$ lines of $\cL_2$ meeting both $\Sigma_1$ and $\Sigma_2$ in one point. 
\end{proof}

\begin{cor}\label{dis}
A line of $\Sigma_2$ is disjoint from $\Sigma_1$.
\end{cor}

From Proposition \ref{sec}, the $(q+1)(q^2+1)$ lines of $\cR_1$ have $q+1$ points in common with $\Sigma_1$ and one point in common with $\Sigma_2$ and there is a subset of $\cL_2$, say $\cR$, consisting of $q^2(q+1)^2(q^2+1)$ lines meeting both $\Sigma_1$ and $\Sigma_2$ in exactly one point. Let $\pi$ be a plane of $\Sigma_1$. From the proof of Proposition \ref{sec} we have that $\pi \cap \Sigma_2$ is a $q$--order subline, say $s$. There are $q+1$ lines of $\cR_1$ contained in $\pi$ and the set $\cD$ of these $q+1$ lines consists of the extended sublines of a dual subconic of $\pi \cap \Sigma_1$; see \cite{Glynn}. Moreover there are $q^2(q+1)$ lines of $\cR$ contained in $\pi$; let $\cE$ be the set of these lines. Note that $\cD \cup \cE$ is the set of lines of $\pi$ having at least a point in common with both $s$ and $\pi \cap \Sigma_1$.      

\begin{lemma}\label{plane2}
If $r_1$ and $r_2$ are two distinct lines of $\cR_1 \cup \cR$ such that $r_1 \cap r_2 \in \cO_3$, then the plane spanned by $r_1$ and $r_2$ is a plane of $\Sigma_1$.
\end{lemma}
\begin{proof}
Since no point of $\cO_3$ lies on a line of $\cR_1$, we have that $r_1, r_2 \in \cR$. Moreover $|\Sigma_1 \cap \cO_3| = |\Sigma_2 \cap \cO_3| = 0$ and hence $r_1 \cap r_2 \notin \Sigma_1 \cup \Sigma_2$. Since $r_i \in \cL_2$, $i = 1,2$, there is a plane of $\Sigma_1$, say $\pi_i$, containing $r_i$, $i = 1,2$. Assume by contradiction that $\sigma = \langle r_1, r_2 \rangle$ is not a plane of $\Sigma_1$. Then $\pi_1 \ne \pi_2$ and $r_1 \cap r_2 \in \pi_1 \cap \pi_2$, a contradiction since two planes of $\Sigma_1$ meet in a line of $\Sigma_1$, which contains no point of $\cO_3$.
\end{proof}

\begin{prop}\label{exists}
There exists a point $P \in \cO_3$ such that $P$ is not contained in a line of $\cR_1 \cup \cR$.
\end{prop}
\begin{proof}
Since no point of $\cO_3$ lies on a line of $\cR_1$, we have that if a point of $\cO_3$ lies on some line of $\cR_1 \cup \cR$ then such a line belongs to $\cR$. From Lemma \ref{plane2}, if $P \in \cO_3$ and if $r_1$ and $r_2$ are two distinct lines of $\cR$ such that $P = r_1 \cap r_2$, then $r_1$ and $r_2$ are contained in a plane of $\Sigma_1$. Therefore, by Theorem \ref{plane1}, through a point of $\cO_3$, there pass $0, 1, 2, 3$ or $q+1$ lines of $\cR$. Fix a plane of $\Sigma_1$, say $\pi$, and let $z_i$ denote the number of points $Q \in \cO_3 \cap \pi$ such that there are $i$ lines of $\cR$ through $Q$, $i = 2, 3, q+1$. Since $S$ is transitive on planes of $\Sigma_1$ and two distinct planes of $\Sigma_1$ have no point of $\cO_3$ in common, we have that there are exactly $z_i(q+1)(q^2+1)$ points of $\cO_3$ such that there are $i$ lines of $\cR$ through $Q$, $i = 2,3, q+1$. In the last case, since $z_{q+1} = q^2$, the points of $\cO_3$ contained in the $q+1$ lines of $\cR$ are exactly $q^2(q+1)(q^2+1)$ and every plane of $\Sigma_1$ contains $q^2$ of these points. Since a line of $\cR$ contains $q^3-q^2$ points of $\cO_3$, it follows that the number of points of $\cO_3$ lying on at least one line of $\cR$ equals
\begin{align*}
(q^3 - q^2) |\cR| - q \cdot z_{q+1} (q+1)(q^2+1) - z_2 (q+1) (q^2+1) - 2 \cdot z_3 (q+1)(q^2+1) \\
= (q+1)(q^2+1) (q^6-q^4-q^3-z_2-2z_3),
\end{align*}
which is smaller than $|\cO_3| = (q+1)(q^2+1) (q^6-q^5-q^4+q^3)$, since $z_2 + 2z_3 > q^5 - 2q^3$, whenever $q \ge 5$, by Proposition \ref{bound}. If $q \le 4$, then some computations performed with Magma \cite{magma} confirm the statement.
\end{proof}

By Proposition \ref{exists} there is a point $P \in \cO_3$ with the property that no line of $\PG(3, q^3)$ having at least a point in common with $\Sigma_1$ and $\Sigma_2$ passes through $P$. Let $\Sigma_3 = P^S$. 

\begin{theorem}\label{Th:mainPG3q^3}
The set $\Sigma_1 \cup \Sigma_2 \cup \Sigma_3$ is a cutting blocking set of $\PG(3, q^3)$ of size $3(q+1)(q^2+1)$.
\end{theorem}
\begin{proof}
First observe that no line of $\PG(3, q^3)$ has at least a point in common with each of the subgeometries $\Sigma_1$, $\Sigma_2$ and $\Sigma_3$. Assume on the contrary that $\ell$ is a line of $\cR_1 \cup \cR$ having a point $R$ in common with $\Sigma_3$. Since $\Sigma_3 = P^S$, there is a projectivity $\gamma$ of $S$ mapping $R$ to $P$ and hence $\ell^\gamma$ is a line of $\cR_1 \cup \cR$ containing $P$, contradicting the fact that no line of $\cR_1 \cup \cR$ passes through $P$. 

Let $\Pi$ be a plane of $\PG(3, q^3)$. Then $|\Pi \cap \Sigma_i| \ge 1$ and hence $|\Pi \cap (\Sigma_1 \cup \Sigma_2 \cup \Sigma_3)| \ge 3$. Suppose by contradiction that $\Sigma_1 \cup \Sigma_2 \cup \Sigma_3$ is not a cutting blocking set of $\PG(3, q^3)$. Hence there is a plane $\Pi$ of $\PG(3, q^3)$ such that the points of $\Pi \cap (\Sigma_1 \cup \Sigma_2 \cup \Sigma_3)$ are on a line, say $r$. 
It follows that $|\Pi \cap (\Sigma_1 \cup \Sigma_2 \cup \Sigma_3)| \ge 3$ and $r$ is a line of $\PG(3, q^3)$ having at least one point in common with each of the subgeometries $\Sigma_1$, $\Sigma_2$ and $\Sigma_3$; a contradiction.
\end{proof}

\begin{prop}
The cutting blocking set $\Sigma_1 \cup \Sigma_2 \cup \Sigma_3$ of $\PG(3, q^3)$ is minimal.
\end{prop}
\begin{proof}
Since $S$ acts transitively on $\Sigma_i$, it is enough to prove that through a point of $\Sigma_i$, there is a plane of $\PG(3, q^3)$ intersecting each of the three relevant subgeometries in exactly one point. Let $r$ be a line of $\cR$. Then $|r \cap \Sigma_1| = |r \cap \Sigma_2| = 1$ and $|r \cap \Sigma_3| = 0$. Since $r \in \cL_2$, from Lemma~\ref{planes}, there are $q^2+1$ planes through $r$ intersecting $\Sigma_1$ in at least $q+1$ points. Similarly, from Lemma \ref{planes}, we deduce that there are at most $q^2+q+1$ planes through $r$ having in common at least $q+1$ points with $\Sigma_2$ and at most $q+1$ planes through $r$ intersecting $\Sigma_3$ in at least $q+1$ points. Taking into account Lemma \ref{intersection}, we have that there are at least $q^3+1 - (q^2+1) - (q^2+q+1) - (q+1) = q^3 - 2 (q^2+q+1)$ planes through $r$ having exactly one point in common with each of the subgeometries $\Sigma_1$, $\Sigma_2$ and $\Sigma_3$. Hence if $q \ge 3$, we are done. If $q = 2$, Magma computations \cite{magma} show the assertion.           
\end{proof}

\begin{prop}
To the set $\Sigma_1 \cup \Sigma_2 \cup \Sigma_3$ there corresponds a $[3(q+1)(q^2+1), 4]_{q^3}$ reduced minimal linear code with possible weights $3q^3+3q^2+3q$, $3q^3+3q^2+2q$, $3q^3+3q^2+q$, $3q^3+3q^2$, $3q^3+2q^2+q$ and $3q^3+2q^2$.
\end{prop}
\begin{proof}
It is enough to observe that, from Lemma \ref{intersection} and Lemma \ref{bruen}, a plane of $\PG(3, q^3)$ can intersect the set $\Sigma_1 \cup \Sigma_2 \cup \Sigma_3$ in $3$, $q+3$, $2q+3$, $3q+3$, $q^2+2q+3$ or $q^2+3q+3$ points. 
\end{proof}

\section{Cutting blocking sets from lines in hyggledy--piggledy arrangement}\label{hp}

In $\PG(3, q)$ let $\ell_1,\ell_2, \ell_3$ three pairwise skew lines. There are $q+1$ lines meeting $\ell_i$, $1 \le i \le 3$, in one point. These lines form a {\em regulus}, say $\cR$ and are contained in a hyperbolic quadric $\cQ^+(3, q)$.  Let $\cR^o$ denote the opposite regulus of $\cR$. It follows that a set of lines of $\PG(3,q)$ in hyggledy--piggledy arrangement has to contain at least four elements. Let $r$ be a line external to $\cQ^+(3,q)$. Then the set $\cB$ consisting of the $4(q+1)$ points of $\ell_1 \cup \ell_2 \cup \ell_3 \cup r$ forms a cutting blocking set, see for instance \cite[Example 9]{FS}, \cite[Theorem 3.7]{DGMP}. 

\begin{prop}\label{solid}
In $\PG(3, q)$, $q > 2$, the cutting blocking set $\cB$ is minimal.
\end{prop}
\begin{proof}
Let $P \in r$ and let $\sigma$ be the plane spanned by $P$ and a line $s$ of $\cR^o \setminus \{\ell_1, \ell_2, \ell_3\}$. Then $\sigma$ meets the hyperbolic quadric $\cQ^+(3, q)$ in two lines, one of which is $s$ and the other one, say $s'$, belongs to $\cR$. Note that $\ell_i \cap \sigma \in s'$, $i = 1,2,3$, and hence $\langle (\cB \setminus \{P\}) \cap \sigma \rangle = s' \ne \sigma$. If $P \in \ell_i$, let $\cQ$ be the hyperbolic quadric obtained by considering the regulus of $\PG(3, q)$ containing the lines $\ell_j, \ell_k$ and $r$, where $\{i, j, k\}$ is a permutation of $\{1,2,3\}$. By repeating the same argument, interchanging $\cQ^+(3, q)$ with $\cQ$, $r$ with $\ell_i$ and $\ell_1, \ell_2, \ell_3$ with $\ell_j, \ell_k, r$, we have that $\cB \setminus \{P\}$ is not a cutting blocking set.  
\end{proof}

\begin{prop}
The code associated with the minimal cutting blocking set of Proposition \ref{solid} is a $[4(q+1), 4]_q$ reduced minimal linear code with weights $3q$ and $4q$ and weight distribution $A_{3q} = 4(q^2-1)$ and $A_{4q} = (q^2-3)(q^2-1)$.
\end{prop}
\begin{proof}
It is enough to observe that there are exactly $4(q+1)$ planes containing one of the four relevant lines and hence meeting $\cB$ in $q+4$ points and that the remaining $(q^2-3)(q+1)$ planes meet $\cB$ in four points.
\end{proof}

In $\PG(r, q)$, let $\cC = \{(1,t, \dots, t^{r-1}, t^r) \;\; | \;\; t \in \GF(q)\} \cup \{(0, 0, \dots, 0, 1)\}$ be the {\em normal rational curve} of $\PG(r, q)$. Then $\cC$ consists of $q+1$ points of $\PG(r, q)$ no $k+2$ of which in a $k$--space of $\PG(r, q)$. Also, for each point $P$ of $\cC$ there is a distinguished line $t_P$ passing through $P$, that is the {\em tangent to $\cC$ at $P$}, where $t_P = \langle P, P' \rangle$, $P' = (0, 1, 2t, \dots, (n-1)t^{n-2}, nt^{n-1})$ if $P \ne (0, 0, \dots, 0, 1)$, and $P' = (0, 0, \dots, 1, 0)$ if $P = (0, 0, \dots, 0, 1)$. Moreover no two  tangent lines to $\cC$ have a point in common (cf. \cite[Lemma 6.31]{HT}). For further properties of the normal rational curve we refer the reader to \cite[Section 6.5]{HT}. The following result has been proved in \cite[Theorem 20]{FS}.

\begin{theorem}[\cite{FS}]
If $p > r$ and $q > 2r - 1$, then arbitrary $2r-1$ distinct tangent lines to $\cC$ constitute a set of lines of $\PG(r, q)$ in higgledy--piggledy arrangement.
\end{theorem}

\begin{remark}
From the construction of Fancsali and Sziklai, there arises a cutting blocking set of $\PG(r, q)$ of size $(2r-1)(q+1)$. However the cutting blocking set so obtained is in general not minimal. For instance in $\PG(4, 11)$ with the aid of Magma \cite{magma} it is possible to see that suitably selecting six  tangent lines to $\cC$, a minimal cutting blocking set of $\PG(4, 11)$ is obtained.  
\end{remark}

\begin{prob}
Determine the minimum number of lines tangent to a normal rational curve in $\PG(r, q)$ such that the set of points covered by these lines forms a minimal cutting blocking set.
\end{prob}

Assume that $p >r$ and that $q > 2r-1$. Let $\bar\cS$ be the set of points covered by $2r-1$ arbitrarily chosen tangent lines  to $\cC$. Let $\cS$ be a pointset such that $\cS \subseteq \bar\cS$ and $\cS$ is a minimal cutting blocking set. 

\begin{prop}\label{prop:iper-normal-curve}
A hyperplane of $\PG(r, q)$ contains at most $r-2$ lines that are tangent to $\cC$.
\end{prop}
\begin{proof}
By induction on $r$. Let $r = 3$. Since no two lines tangent to $\cC$ have a point in common, a plane of $\PG(3, q)$ contains at most one tangent line to $\cC$. Assume that the result holds true for $r-1$. Let $\cC$ be a normal rational curve of the projective space $\PG(r, q)$ equipped with homogeneous projective coordinates $X_1,\dots,X_{r+1}$. Suppose by contradiction that a hyperplane $H$ of $\PG(r, q)$ contains $r-1$ lines tangent to $\cC$. Denote by $P$ a point of $\cC \cap H$, where $t_P$, the line tangent to $\cC$ at $P$, is contained in $H$. Let $\Gamma$ be a hyperplane of $\PG(r, q)$ such that $P \notin \Gamma$. From Lemma \cite[Theorem 6.30]{HT}, we may assume that $P = (0, 0, \dots, 0, 1)$. By projecting the $q$ points of $\cC \setminus \{P\}$ and the line $t_{P}$  from $P$ onto $\Gamma: X_{r+1} = 0$, we obtain the normal rational curve $\cC' = \{(1,t, \dots, t^{r-1}, 0) \;\; | \;\; t \in \GF(q)\} \cup \{(0, 0, \dots, 1, 0)\}$ of $\Gamma$. On the other hand, by projecting from $P$ onto $\Gamma$ the line tangent to $\cC$ at a point $R$, $R \ne P$, we get the line $\langle \bar{R}, \bar{R}' \rangle$ that is tangent to $\cC'$ at $\bar{R}$, where $\bar{R} = \left( 1,t, \dots, t^{r-1}, 0 \right)$ and $\bar{R}' = \left(0, 1, \dots, (r-1)t^{r-2}, 0 \right)$. Observe that by projecting the $r-1$ tangent lines to $\cC$ contained in $H$ we get $r-2$ lines that are tangent to $\cC'$ and are contained in the $(r-2)$--space $H \cap \Gamma \subset \Gamma$, a contradiction. 
\end{proof}

\begin{prop}
Let $p >r$ and $q > 2r-1$. The code associated with the minimal cutting blocking set $\cS$ is a $[|\cS|, r+1]_q$ reduced minimal linear code, where $|\cS| \le (2r-1)(q+1)$ and minimum distance $d \ge |\cS| - (r-2)q-2r+1$.
\end{prop}
\begin{proof}
The statement follows from Proposition \ref{prop:iper-normal-curve}.
\end{proof}

\subsection{Lines of $\PG(5, q)$ in higgledy--piggledy arrangement}

Let $\PG(2, q^2)$ be the Desarguesian projective plane. Its underlying vector space $V(3, q^2)$  can be considered as a $6$--dimensional vector space $V(6, q)$ via the inclusion $\GF(q) \subset \GF(q^2)$. Each point in $\PG(2, q^2)$ corresponds to a $1$--dimensional vector subspace in $V(3, q^2)$ which in turn corresponds to a $2$--dimensional vector subspace in $V(6, q)$, i.e., a line of $\PG(5, q)$. Extending this map from (subsets of) points of $\PG(2, q^2)$ to subsets of points of $\PG(5, q)$ we obtain a map $\phi: \PG(2, q^2) \rightarrow \PG(5, q)$, called {\em field reduction map}.

The set $\cD=\{\phi(P) \;|\; P \in \PG(2, q^2)\}$ is a Desarguesian line--spread of $\PG(5, q)$. The incidence structure whose points are the elements of $\cD$ and whose lines are the solids of $\PG(5, q)$ joining two distinct elements of $\cD$, is isomorphic to $\PG(2, q^2)$. A degenerate Hermitian curve of rank $2$ of $\PG(2, q^2)$ is a cone having as vertex a point and as base a Baer subline. In what follows we will refer to a degenerate Hermitian curve of rank $2$ as a {\em degenerate Hermitian curve}. Let $\Pi$ be a solid of $\PG(5, q)$ and let $\phi^{-1}(\Pi) = \{P \in \PG(2, q^2) \; : \; |\phi(P) \cap \Pi| \ge 1\}$. Then either $\Pi$ contains $q^2+1$ lines of $\cD$ and $\phi^{-1}(\Pi)$ is a line of $\PG(2, q^2)$ or there is exactly one line $\ell$ of $\cD$ in $\Pi$ and $\phi^{-1}(\Pi)$ is a degenerate Hermitian curve $\cH$ of $\PG(2, q^2)$. Note that to such a degenerate Hermitian curve there correspond $q+1$ solids of $\PG(5, q)$, say $\Pi_i$, $1 \le i \le q+1$, such that $\Pi_i \cap \Pi_j = \ell$, if $i \ne j$, and $\bigcup_{i= 1}^{q+1} \Pi_i = \bigcup \{\phi(P) \;\; | \;\; P \in \cH\}$. 

\begin{lemma}\label{preliminary}
Let $\cP$ be a set of points of $\PG(2, q^2)$ not on a line of $\PG(2, q^2)$ and not on a degenerate Hermitian curve of $\PG(2, q^2)$. Then $\cL = \{\phi(P) \;\; | \;\; P \in \cP\}$ is a set of lines of $\PG(5, q)$ in higgledy--piggledy position.  
\end{lemma}
\begin{proof}
Assume by contradiction that the lines of $\cL$ are not in higgledy--piggledy position. Then there would exist a solid $\Pi$ of $\PG(5, q)$ meeting every line of $\cL$ and $\cP$ would be contained in $\phi^{-1}(\Pi)$, a contradiction. 
\end{proof}

In $\PG(2, q^2)$, consider the set $\bar \cP$ consisting of the following four points: 
$$
P_1 = (1, 0, 0),\ P_2 = (0, 1, 0),\ P_3 = (0, 0, 1),\ P_4 = (1, 1, 1).
$$
There is a unique Baer subplane of $\PG(2, q^2)$ containing the four points of $\bar \cP$, namely the canonical Baer subplane $\pi$. 
\begin{lemma}
There are exactly $q^3+4q^2+1$ degenerate Hermitian curve containing the four points of $\bar \cP$.
\end{lemma}
\begin{proof}
Let $\cH$ be a degenerate Hermitian curve containing the four points of $\bar \cP$ and let $V$ be its vertex. Assume first that $V \in \pi$. If $V \not\in \{(0, 1, 1), (1, 0, 1), (1, 1, 0)\}$, then at least three of the lines of $\cH$ meet $\pi$ in a Baer subline. In this case  the $q+1$ lines of $\cH$ have all $q+1$ points in common with $\pi$. Hence $\pi \subset \cH$ and $\cH$ is uniquely determined by $V$. If $V \in \{(0, 1, 1), (1, 0, 1), (1, 1, 0)\}$, then either $\pi \subset \cH$ and $\cH$ is uniquely determined or $\cH$ intersects $\pi$ in the two lines through $V$ whose union contains the four points of $\bar \cP$ and there are $3q$ choices for $\cH$. Assume now that $V \notin \pi$. In this case $\pi \not\subset \cH$ and $\cH \cap \pi$ is a quadric. Since the quadric $\cH \cap \pi$ contains at least four points no three on a line, we have that either $\cH \cap \pi$ consists of two intersecting Baer sublines of $\pi$ containing the four points of $\pi$ or $\cH \cap \pi$ is a non--degenerate Baer conic $\bar{C} \subset \pi$. In the former case $V$ lies on a line secant to $\bar \cP$ and for a fixed $V$ there is a unique $\cH$. Hence there are $6(q^2-q)$ degenerate Hermitian curves of this type. In the latter case, from \cite[Corollary 6.2]{DD}, we have that $V \in \cC$, where $\cC$ is the unique non--degenerate conic of $\PG(2, q^2)$ such that $\cC \cap \pi = \bar{\cC}$, and for a fixed $V$ there is a unique $\cH$. Since there are $q-2$ of such non--degenerate conics, it follows that there are $(q-2)(q^2-q)$ degenerate Hermitian curves of this type.    
\end{proof}

\noindent The degenerate Hermitian curves of the previous lemma are listed and described below.
\begin{itemize}
    \item $q^2+q+1$ degenerate Hermitian curves of type
\begin{align}
 & X_1^qX_2-X_1X_2^q = 0, {\nonumber} \\
 & a (X_1^qX_2-X_1X_2^q) - (X_1^qX_3-X_1X_3^q) = 0, \; a \in \GF(q), \label{type1} \\ 
 & b (X_1^qX_2-X_1X_2^q) - a (X_1^qX_3-X_1X_3^q) + X_2^qX_3-X_2X_3^q = 0, \; a, b \in \GF(q). {\nonumber} 
\end{align} They contain $\pi$ and their vertices lie in $\pi$.
\item $3q$ degenerate Hermitian curves of type
\begin{align}
& X_1^qX_3 - \alpha X_1X_3^q - X_2^qX_3 + \alpha X_2X_3^q = 0, {\nonumber} \\
& X_1^qX_2 - \alpha X_1X_2^q + \alpha X_2^qX_3 - X_2X_3^q = 0, \label{type2} \\
& X_1^qX_2 - \alpha X_1X_2^q - X_1^qX_3 + \alpha X_1X_3^q = 0, \; \alpha \in \GF(q^2) \setminus \{1\}, \; \alpha^{q+1} = 1. {\nonumber} 
\end{align} They meet $\pi$ in $2q+1$ points and their vertices belong to $\{(0, 1, 1), (1, 0, 1), (1, 1, 0)\}$.
\item $6(q^2-q)$ degenerate Hermitian curves of type
\begin{align}
& (1 - \alpha^q)\alpha X_1^qX_2 - (1 - \alpha) \alpha^q X_1X_2^q - (1 - \alpha^q) X_1^qX_3 + (1 - \alpha) X_1X_3^q = 0, {\nonumber} \\
& (1 - \alpha)\alpha^q X_1^qX_2 - (1 - \alpha^q) \alpha X_1X_2^q + (1 - \alpha^q) X_2^qX_3 - (1 - \alpha) X_2X_3^q = 0, {\nonumber} \\
& (1 - \alpha)\alpha^q X_1^qX_3 - (1 - \alpha^q) \alpha X_1X_3^q - (1 - \alpha) X_2^qX_3 + (1 - \alpha^q) X_2X_3^q = 0, \label{type3} \\
& \alpha^{q+1} (X_1^qX_2 - X_1X_2^q) - \alpha^q X_1^qX_3 + \alpha X_1X_3^q + \alpha^q X_2^qX_3 - \alpha X_2X_3^q = 0, {\nonumber} \\
& \alpha^q X_1^qX_2 - \alpha X_1X_2^q - \alpha^{q+1} (X_1^qX_3 - X_1X_3^q) + \alpha X_2^qX_3 - \alpha^q X_2X_3^q = 0, {\nonumber} \\
& \alpha X_1^qX_2 - \alpha^q X_1X_2^q - \alpha X_1^qX_3 + \alpha^q X_1X_3^q + \alpha^{q+1} (X_2^qX_3 - X_2X_3^q) = 0, \; \alpha \in \GF(q^2) \setminus \GF(q). {\nonumber}
\end{align}
 They meet $\pi$ in $2q+1$ points and their vertices are not in $\pi$ and on some line secant to $\bar \cP$.
 \item $(q-2)(q^2-q)$ degenerate Hermitian curves of type
\begin{align}
& (1 - \delta) t^q X_1^qX_2 - (1 - \delta) t X_1X_2^q - (1 - \delta t) t^q X_1^qX_3 {\nonumber}\\
&\qquad + (1 - \delta t^q) t X_1X_3^q + (1 - \delta t) X_2^qX_3 - (1 - \delta t^q) X_2X_3^q = 0, {\nonumber} \\
& \; \delta \in \GF(q) \setminus \{0, 1\}, t \in \GF(q^2) \setminus \GF(q). \label{type4}
\end{align}
 They meet $\pi$ in $q+1$ points and their vertices are not in $\pi$ and on $q-2$ non--degenerate conics of $\PG(2, q^2)$.
\end{itemize}

\begin{lemma}\label{six}
There exists at least a degenerate Hermitian curve of $\PG(2, q^2)$ passing through six points.
\end{lemma}
\begin{proof}
Consider a set of six points of $\PG(2, q^2)$. If there were not four points out of the six points no three on a line, then these six points would lie on at most three concurrent lines and hence on at least a degenerate Hermitian curve of $\PG(2, q^2)$. Thus we may assume that there are four points no three of them on a line and, by the action of $\PGL(3, q^2)$, that these four points are those of $\bar \cP$. Let $P$ and $Q$ be the remaining two points. If at least one of the points $P$ and $Q$ lies in $\pi$, then there will be at least one degenerate Hermitian curve of type \eqref{type1} containing the six points. Assume that $P, Q \notin \pi$ and let $\ell_P$ and $\ell_Q$ be the lines of $\PG(2, q^2)$ containing $P$ and $Q$, respectively, and meeting the Baer subplane $\pi$ in $q+1$ points. Then $|\ell_P \cap \ell_Q \cap \pi| \ge 1$. In this case the degenerate Hermitian curve of type \eqref{type1} having as vertex a point of $\ell_P \cap \ell_Q$ and containing $\ell_P$ and $\ell_Q$ will contain the six points. 
\end{proof}
\begin{cor}
A set of points of $\PG(2, q^2)$ not contained in a degenerate Hermitian curve of $\PG(2, q^2)$ has at least seven points.
\end{cor}

\subsubsection{Seven lines of $\PG(5, q)$ in higgledy--piggledy arrangement}

In this section we obtain a set of seven points of $\PG(2, q^2)$ not contained in a degenerate Hermitian curve of $\PG(2, q^2)$ and hence, as a by product, a set of seven lines of $\PG(5, q)$ in higgledy--piggledy arrangement. 

Let $G$ be the group of projectivities of order three generated by 
$$
\begin{pmatrix}
0 & 0 & 1 \\
1 & 0 & 0 \\
0 & 1 & 0 \\
\end{pmatrix}.
$$
Then the group $G$ fixes $P_4$ and permutes in a unique orbit the remaining three points of $\bar \cP$. Note that since the group $G$ fixes $\bar \cP$, then the set of degenerate Hermitian curves of each of the four types described above is preserved by $G$.

Let $P_5 = (1, x, \xi x)$, with $x \in \GF(q) \setminus \{0\}, \xi \in \GF(q^2)$ and let 
\begin{equation}\label{set}
\cP_{x, \xi} = {\bar \cP} \cup P_5^G, \mbox{ where } P_5^G = \{(1, x, \xi x), (\xi x, 1, x), (x, \xi x, 1)\}. 
\end{equation}
Our aim is to determine the existence of $x$ and $\xi$ such that $|\cP_{x, \xi}| = 7$ and no degenerate Hermitian curve containing $\bar \cP$ contains $\cP_{x, \xi}$. Let $\cH$ be a degenerate Hermitian curve containing ${\bar \cP}$. Assume first that $\cH$ is defined by one of the Equations of \eqref{type1} for some $a, b \in \GF(q)$. Straightforward calculations show that if 
\begin{equation}\label{cond1}
x^3 \ne 1 \mbox{ and } \xi^q \ne \xi, 
\end{equation}
then $\cP_{x, \xi} \not\subset \cH$. Let $\cH$ be defined by one of the equations of \eqref{type2} for some $\alpha \in \GF(q^2) \setminus \{1\}$, with $\alpha^{q+1} = 1$ or by one of the equations of \eqref{type3} for some $\alpha \in \GF(q^2) \setminus \GF(q)$. In this case it can be seen that if either
\begin{equation}\label{cond2}
x \ne \pm 1 \mbox{ and } \xi \notin \left\{\frac{\alpha-x}{x(\alpha-1)} \;\; | \;\; \alpha^{q+1} = 1, \alpha \ne \pm 1\right\} \cup \left\{\frac{-1}{\alpha^q(1+x)} \;\; | \;\; \alpha^{q+1}+\alpha^q+\alpha = 0, \alpha \ne 0, -2 \right\},
\end{equation}
or   
\begin{equation}\label{cond2bis}
q \mbox{ is odd, } x = - 1 \mbox{ and } \xi \notin \left\{\frac{1 + \alpha}{1 - \alpha} \;\; | \;\; \alpha^{q+1} = 1, \alpha \ne \pm 1\right\},
\end{equation}
then $\cP_{x, \xi} \not\subset \cH$. Assume now that either Conditions \eqref{cond1} and \eqref{cond2} or \eqref{cond1} and \eqref{cond2bis} hold true and that $\cH$ is given by Equation \eqref{type4}, for some $\delta \in \GF(q) \setminus \{0, 1\}$ and $t \in \GF(q^2) \setminus \GF(q)$. 
Put 
\begin{align*}
& A := A(\delta, t) = t^q \left(1- \delta - (1 - \delta t) x \right), \\
& B := B(\delta, t) = 1- \delta t - (1 - \delta) t x, \\
& C := C(\delta, t) = \delta (t^q -t), \\
& D := D(\delta, t) = t - t^q, \\ 
& E := E(\delta, t) = (t - x) (1- \delta t^q).
\end{align*}
Suppose that $\cP_{x, \xi} \subset \cH$, then the following equations are satisfied.
\begin{equation}\label{sys0}
\left\{
\begin{aligned}
& \xi^q E - \xi E^q - C - D = 0 \\
& \xi^q A  - \xi A^q + C = 0 \\ 
& \xi^q B - \xi B^q + D = 0. 
\end{aligned}
\right. 
\end{equation}
Note that $A+B+E = (x-1) (\delta t^{q+1} + (\delta - 1) (t^q + t) - 1)\in\GF(q)$ and hence summing up the three equations we get $(\xi^q - \xi)(x-1) (\delta t^{q+1} + (\delta - 1) (t^q + t) - 1) = 0$. Moreover $A \ne 0$, since $t \in \GF(q^2) \setminus \GF(q)$. Then the previous system can be rewritten as follows:
\begin{equation}\label{sys1}
\left\{
\begin{aligned}
& \delta t^{q+1} + (\delta - 1) (t^q + t) - 1 = 0 \\
& \xi^q = (\xi A^q - C)/A \\ 
& \xi \left(A^q B - A B^q\right) + A D - B C = 0. \\
\end{aligned}
\right. 
\end{equation}

Note that $A^q B - A B^q=0$ yields $A D - B C=0$. Also, $(A^qB - AB^q, AD - BC) = (0, 0)$ if and only if $A = - \delta B$, i.e., 
$$
\delta x t^{q+1} + (1 - \delta - x) t^q + \delta \left( (\delta - 1) x - \delta \right) t + \delta = 0.
$$ 
\begin{lemma}\label{Lemma1}
For every $x \in \GF(q) \setminus \{0\}$ such that $x^3 \ne 1$, there are no $\delta \in \GF(q) \setminus \{0, 1\}$ and $t \in \GF(q^2) \setminus \GF(q)$ such that 
$$
\left\{
\begin{aligned}
& \delta t^{q+1} + (\delta - 1) (t^q + t) - 1 = 0 \\
& \delta x t^{q+1} + (1 - \delta - x) t^q + \delta \left( (\delta - 1) x - \delta \right) t + \delta = 0. \\
\end{aligned}
\right. 
$$
\end{lemma}
\begin{proof}
By multiplying the first equation by $x$ and by subtracting the second equation from it, we get 
\begin{equation} \label{sub_sys1}
\left\{
\begin{aligned}
& \delta t^{q+1} + (\delta - 1) (t^q + t) - 1 = 0 \\
& \left((1 + x) \delta - 1\right) t^q + \left(\delta^2(1 - x) + 2 \delta x - x \right) t - x - \delta = 0. \\
\end{aligned}
\right. 
\end{equation}
If $x \ne -1$ and $\delta = 1/(x+1)$, then from the second equation of \eqref{sub_sys1} we have that $t = \frac{(x+1)^3}{1 - x^3} \in \GF(q)$, a contradiction. Hence, from \cite[Corollary 1.24]{H1}, the second equation of \eqref{sub_sys1} admits solutions in $t$ if and only if $\frac{\delta^2(1 - x) + 2 \delta x - x}{1 - (1 + x) \delta} = - 1$, i.e., $\delta^2 - \delta + 1 = 0$. Therefore assume that $q \not\equiv - 1 \pmod{3}$ and $\delta^2 = \delta - 1$. In this case \eqref{sub_sys1} reads
$$
\left\{
\begin{aligned}
& t^{q+1} = - \delta \\
& t^q + t = 1 - \delta, \\
\end{aligned}
\right. 
$$
which in turn is equivalent to the following quadratic equation in $t$
\begin{equation}\label{sub_sys2}
t^2 + (\delta - 1) t - \delta = 0.
\end{equation}
If $q$ is odd, \eqref{sub_sys2} has one or two solutions according as $\delta = -1$ or $\delta \ne -1$ and these solutions are in $\GF(q)$. If $q$ is even, note that $\frac{\delta}{\delta^2+1} = \frac{\delta^2 + 1}{\delta^2 + 1} = 1$, where ${\rm Tr}_{q|2} (1) = 0$, since $q \not\equiv - 1 \pmod{3}$. It follows that \eqref{sub_sys2} has two solutions and these solutions are in $\GF(q)$. The proof is now complete.
\end{proof}

By Lemma \ref{Lemma1} we can assume $A^q B - A B^q\neq0$ and so \eqref{sys1} reads
\begin{equation}\label{sys2}
\left\{
\begin{aligned}
& \delta t^{q+1} + (\delta - 1) (t^q + t) - 1 = 0 \\
& \xi = \frac{B C - A D}{{A}^q B - A {B}^q}. \\
\end{aligned}
\right. 
\end{equation}
Note that if $\omega =(\overline{x},\overline{\xi},\overline{\delta}, \overline{t})$ is a solution of \eqref{sys2} then 
$$\omega^{\prime}=\left(\overline{x},\overline{\xi},\frac{\overline{\delta}-1}{\overline{\delta}}, \frac{1-\overline{\delta}\overline{t}}{(1-\overline{\delta})\overline{t}}\right)\qquad  \textrm{and} \qquad \omega^{\prime\prime}=\left(\overline{x},\overline{\xi},\frac{1}{1-\overline{\delta}}, \frac{1-\overline{\delta}}{1-\overline{\delta}\overline{t}}\right)$$
are solutions of \eqref{sys2} too. Also, $\omega=\omega^{\prime}$ yields $(1-\delta) t^2 + \delta t - 1 = \left((1 - \delta) t + 1\right) (t - 1) = 0$ and $t \in \GF(q)$, a contradiction. Similarly, if $\omega=\omega^{\prime\prime}$ or $\omega^{\prime}=\omega^{\prime\prime}$, then $\delta t^2 - t + 1 - \delta = (\delta t - 1 + \delta)(t - 1) = 0$ or $\delta^2 t^2 - (\delta^2 + 1) t + 1 = (\delta^2 t - 1)(t - 1) = 0$ and $t \in \GF(q)$, a contradiction. 
Let 
$$\Sigma = \left\{(\delta, t) \in \left(\GF(q) \setminus \{0, 1\}, \GF(q^2) \setminus \GF(q)\right) \;\; | \;\; \delta t^{q+1} + (\delta - 1) (t^q + t) - 1 = 0 \right\}.$$ It is readily seen that $|\Sigma|\leq q^2$. 


\begin{theorem}\label{main}
There exist $(x, \xi) \in \GF(q) \setminus \{0, 1\} \times \GF(q^2) \setminus \GF(q)$ such that $\cP_{x, \xi}$ is a set of seven points and no degenerate Hermitian curve of $\PG(2, q^2)$ contains it.
\end{theorem}
\begin{proof}
Fix $x \in \GF(q) \setminus \{0\}$, with $x^3 \ne 1$. We show the existence of an element $\xi \in \GF(q^2) \setminus \GF(q)$ such that no degenerate Hermitian curve containing $\bar \cP$ contains $\cP_{x, \xi}$, where $\cP_{x, \xi}$ is defined as in \eqref{set}. Since Condition \eqref{cond1} is satisfied, taking into account Condition \eqref{cond2} or \eqref{cond2bis},  we have that the number of values that $\xi$ cannot assume is at most $\frac{|\Sigma|}{3}-2q$ if $q$ is even and $\frac{|\Sigma|}{3}-2(q-1)$ if $q$ is odd. Indeed,
$$
\left| \left\{\frac{\alpha-x}{x(\alpha-1)} \;\; | \;\; \alpha^{q+1} = 1, \alpha \ne \pm 1\right\}\right| = 
\begin{cases}
q-1 & \mbox{ if } q \mbox{ is odd, } \\
q & \mbox{ if } q \mbox{ is even, }
\end{cases}
$$
and, if $x \ne -1$,  
$$
\left| \left\{\frac{-1}{\alpha^q(1+x)} \;\; | \;\; \alpha^{q+1}+\alpha^q+\alpha = 0, \alpha \ne 0, -2 \right\} \right| = 
\begin{cases}
q-1 & \mbox{ if } q \mbox{ is odd, } \\
q & \mbox{ if } q \mbox{ is even. }
\end{cases}
$$
Since in both cases the number of forbidden values for $\xi$ is less than $q^2-q$, we have the assertion. 
\end{proof}

\begin{prop}\label{arc}
Let $\cP_{x, \xi}$ be defined as in Theorem \ref{main}. No three points of $\cP_{x, \xi}$ are on a line of $\PG(2, q^2)$.
\end{prop}
\begin{proof}
Assume by contradiction that three points of $\cP_{x, \xi}$ are on a line of $\PG(2, q^2)$. Then the determinant of the $3 \times 3$ matrix whose rows are the coordinates of these points must be zero. Straightforward computations show that either $\xi$ belongs to $\GF(q)$ or $x \in \{0, 1\}$ and we have a contradiction from \eqref{cond1}, or one of the following possibilities occurs:
\begin{align}
& x^2 \xi^2-x(x+1) \xi + x^2 - x + 1 = 0, \label{eq_arc1} \\
& \xi^2=\frac{1}{x}. \label{eq_arc2}
\end{align}

If \eqref{eq_arc1} were satisfied, then either $q \not\equiv -1 \pmod{3}$ and $\xi$ should be in $\GF(q)$ or $q \equiv -1 \pmod{3}$ and $\xi$ should be equal to $\frac{\alpha - x}{x(\alpha - 1)}$, for some $\alpha \in \GF(q^2) \setminus \GF(q)$, with $\alpha^2 - \alpha + 1 = 0$ (and hence $\alpha^{q+1} = 1$). From \eqref{cond1}, \eqref{cond2}, \eqref{cond2bis}, in both cases we have a contradiction.  

Suppose that \eqref{eq_arc2} is satisfied. Then $q$ has to be odd and $x$ is non--square in $\GF(q)$. Let $x \ne -1$, otherwise $q \equiv -1 \pmod{4}$ and as before $\xi = \frac{1 + \alpha}{1 - \alpha}$, for some $\alpha \in \GF(q^2) \setminus \GF(q)$, with $\alpha^2 = -1$ (and hence $\alpha^{q+1} = 1$). We will show that there exists a degenerate Hermitian curve given by \eqref{type4} containing $\cP_{x, \xi}$ and hence contradicting Theorem \ref{main}. Note that $\xi^q = -\xi$. From \eqref{sys0}, this implies that $A + A^q \ne 0$ and $B + B^q \ne 0$, otherwise $C = 0$ or $D = 0$, contradicting the fact that $t \in \GF(q^2) \setminus \GF(q)$ and $\delta \ne 0$. From the second equation of \eqref{sys1}, we have that $\xi = \frac{C}{A+A^q}$, whereas from the third equation of \eqref{sys1}, we obtain $A + A^q + \delta (B + B^q) = 0$. Taking into account the first equation of \eqref{sys1}, we get
$$
A+A^q=(t+t^q)(1-\delta-2x)+2\delta t^{q+1}x=2x + (t + t^q)(1 - \delta + x - 2 \delta x).
$$
Hence $\xi = \frac{C}{A+A^q} = \frac{D}{B+B^q}$, that is
\begin{equation}\label{eq_arc3}
    \xi = \frac{t - t^q}{2 - (t + t^q)(x - \delta x + \delta)} = \frac{-\delta (t - t^q)}{2x + (t + t^q)(1 - \delta + x - 2 \delta x)}.
\end{equation}
Equation \eqref{eq_arc3} gives 
$$
2 (x + \delta) + (t + t^q) F(\delta) = 0,
$$
where $F(\delta) = \delta^2 (x-1) - \delta (3x+1) + x+1$. 

Let $x = -\delta$. Thus $F(\delta) = F(-x) = (x+1)(x^2+x+1) \ne 0$, since $x \ne -1$ and $x^3 \ne 1$. It follows that $t^q = -t$ and, from the first equation of \eqref{sys1}, $t^{q+1}=-\frac{1}{x}$. Therefore $\xi = \frac{t-t^q}{2} = t$, where $t^2 = \frac{1}{x}$ and $(x, \xi, -x, \xi)$ is a solution of \eqref{sys2}. 

Let $x \ne -\delta$. Thus $F(\delta) \ne 0$. Moreover 
\begin{align}
& t + t^q = \frac{- 2 (x + \delta)}{F(\delta)}, \label{eq_arc_t} \\
& t^{q+1} = \frac{2 (\delta-1) (x + \delta)}{\delta F(\delta)} + \delta^{-1}. \label{eq_arc_n}
\end{align}
Observe that, from \eqref{eq_arc_t}, $t - t^q = \frac{2 (x + \delta)}{F(\delta)} + 2 t$ and by substituting it together with \eqref{eq_arc_t} into Equation \eqref{eq_arc3}, we have 
\begin{equation}\label{eq_arc4}
x t^2 F(\delta)^2 + 2 t x F(\delta) (x + \delta) + x(x+\delta)^2 = \left( F(\delta) + (x + \delta)(x - \delta x + \delta)\right)^2.
\end{equation}
From \eqref{eq_arc_t}, by using \eqref{eq_arc_n}, it follows that 
\begin{equation}\label{eq_arc5}
t^2 = - \frac{2 t \delta (x + \delta) + 2 (\delta-1) (x + \delta) + F(\delta)}{\delta F(\delta)}.  
\end{equation}
By substituting \eqref{eq_arc5} in Equation \eqref{eq_arc4}, after some calculations, we get 
\begin{equation}\label{eq_arc6}
    \delta^{-1} x (x + \delta) (x^2 - 1) (\delta - 1) \left(\delta - \frac{x+1}{x}\right) \left(\delta - \frac{1}{x + 1}\right) = 0,
\end{equation}
that is $\delta = \frac{x+1}{x}$ or $\delta = \frac{1}{x+1}$. In the former case $\left(x, \xi, \frac{x+1}{x}, \frac{x \pm \sqrt{x}}{x+1}\right)$ is a solution of \eqref{sys2}, whereas if the latter case occurs, then $\left(x, \xi, \frac{1}{x+1}, \frac{(x+1)(-1 \pm \sqrt{x})}{x-1}\right)$ is a solution of \eqref{sys2}. 
\end{proof}

By construction the seven points obtained are left invariant by a group of order three.

\begin{theorem}\label{seven}
Let $\cP_{x, \xi}$ be defined as in Theorem \ref{main}. The set $\phi(\cP_{x, \xi}) = \{\phi(P) \;\; | \;\; P \in \cP_{x, \xi}\}$ consists of seven lines of $\PG(5, q)$ in higgledy--piggledy arrangement.
\end{theorem}

Let $\cB$ be the set of $7(q+1)$ points of $\PG(5, q)$ on the seven mutually skew lines $\ell_i$, $1 \le i \le 7$, constructed in the previous theorem. 

\begin{prop}\label{Prop:mainPG5q}
The set $\cB$ is a minimal cutting blocking set of $\PG(5, q)$.
\end{prop}
\begin{proof}
We only need to show the minimality of $\cB$. Let $P$ be a point of $\cB$ and let $\ell_k$ be the line containing $P$, for a fixed $k \in \{1, \ldots, 7\}$. By Lemma \ref{six}, there are $q+1$ solids of $\PG(5, q)$, $\Pi_i$, $1 \le i \le q+1$, such that each of the lines $\ell_j$, $1 \le j \le 7$, $j \ne k$, has at least one point in common with $\Pi_i$, $1 \le i \le q+1$. Moreover, from Theorem \ref{seven}, $|\ell_k \cap \Pi_i| = 0$, $1 \le i \le q+1$. Let $\Gamma_i$ be the hyperplane spanned by $\Pi_i$ and $P$. Then there are at most six indices $i_1, \dots, i_6$ such that $\Gamma_{i_r} \cap \cB \not \subset \Pi_{i_r}$. Hence if $q > 5$, there exists a hyperplane $\Gamma_{i'}$, $1 \le i' \le q+1$, $i' \not\in \{i_1, \dots, i_6\}$, such that $\Gamma_{i'} \cap \cB \subset \Pi_{i'}$. If $q \le 5$, then some computations performed with Magma \cite{magma} confirm the statement.   
\end{proof}

\begin{prop}
To the set $\cB$ there corresponds a $[7(q+1), 6]_q$ reduced minimal linear code with weights $5q$, $6q$ and $7q$ and weight distribution $A_{5q} = 21(q^2-1)$, $A_{6q} = 7(q^2-5)(q^2-1)$ and $A_{7q} = (q^4-6q^2+15)(q^2-1)$. 
\end{prop}
\begin{proof}
A hyperplane $\Gamma$ of $\PG(5, q)$ contains exactly $q^2+1$ lines of the Desarguesian line--spread $\cD$ of $\PG(5, q)$ and they are contained in a unique solid, say $\Pi$, where $\phi^{-1}(\Pi)$ is a line, say $\ell$, of $\PG(2, q^2)$. From Proposition \ref{arc}, $|\ell \cap \cP| \le 2$ and hence at most two lines out of the seven lines $\{\ell_i \;\; | \;\; 1 \le i \le 7\}$ are contained in $\Gamma$. In particular, there are $21$ lines of $\PG(2, q^2)$ meeting $\cP$ in two points,  $7(q^2-5)$ lines of $\PG(2, q^2)$ intersecting $\cP$ in one point and the remaining $q^4-6q^2+15$ lines do not share any point with $\cP$. Hence there are $21(q+1)$ hyperplanes $\Gamma$ of $\PG(5, q)$ such that $|\Gamma \cap \cB| \le 2(q+1) + 5$, $7(q^2-5)(q+1)$ hyperplanes $\Gamma$ of $\PG(5, q)$ such that $|\Gamma \cap \cB| = (q+1) + 6$ and $(q^4-6q^2+15)(q+1)$ hyperplanes $\Gamma$ of $\PG(5, q)$ such that $|\Gamma \cap \cB| = 7$.
\end{proof}


\smallskip
{\footnotesize
\noindent\textit{Acknowledgments.}
This work was supported by the Italian National Group for Algebraic and Geometric Structures and their Applications (GNSAGA-- INdAM).
}

\end{document}